\documentclass[12pt]{amsart}
\usepackage[T1]{fontenc}

\newtheorem{Theorem}{Theorem}[section]
\usepackage[colorlinks,citecolor=blue]{hyperref}
\newtheorem{theorem}[Theorem]{Theorem}
\newtheorem{lemma}[Theorem]{Lemma}
\newtheorem{proposition}[Theorem]{Proposition}

\newtheorem{corollary}[Theorem]{Corollary}
\newtheorem{definition}[Theorem]{Definition}
\newtheorem{remark}[Theorem]{Remark}
\newtheorem{conjecture}[Theorem]{Conjecture}

\usepackage{amssymb}

\usepackage[backend=biber,style=alphabetic]{biblatex}
\bibliography{bib}

\usepackage{geometry}
\usepackage[all]{xy}

\newcommand{\ab}{\mathrm{ab}}
\newcommand{\loc}{\mathrm{loc}}

\newcommand{\CT}{Colliot-Th{\'e}l{\`e}ne}
\newcommand{\BM}{Brauer--Manin}
\newcommand{\BB}{\mathrm{BB}}
\newcommand{\chow}{\mathrm{CH}}

\newcommand{\br}{\mathrm{Br}}
\newcommand{\brgp}[1]{\mathrm{Br}(#1)}
\newcommand{\brgpf}[2]{\mathrm{Br}_{#1}(#2)}
\newcommand{\brgpu}[1]{\mathrm{Br}_1(#1)}
\newcommand{\brgpa}[1]{\mathrm{Br}_a(#1)}

\newcommand{\pic}[1]{\mathrm{Pic}(#1)}
\newcommand{\cor}[1]{\mathrm{Cor}_{#1}}
\newcommand{\wres}[1]{\mathrm{Res}_{#1}}
\newcommand{\zcycle}[1]{\mathrm{Z}_0^{#1}}
\newcommand{\zcyclef}[2]{\mathrm{Z}_{0,#2}^{#1}}
\newcommand{\prodf}[1]{\mathop{\prod}\limits_{#1}}

\newcommand{\sumf}[1]{\mathop{\sum}\limits_{#1}}
\newcommand{\inv}{{\mathrm{inv}}}
\newcommand{\reff}[1]{{(\ref{#1})}}

\newcommand{\spec}[1]{\mathrm{Spec}(#1)}
\newcommand{\cohomorf}[2]{\mathrm{H}^{#1}(#2)}
\newcommand{\cohomorh}[2]{\mathbb{H}^{#1}(#2)}
\newcommand{\cohomorab}[2]{\mathrm{H}_{\ab}^{#1}(#2)}
\newcommand{\cohomoret}[2]{\mathrm{H}_\et^{#1}(#2)}

\newcommand{\et}{{\text{\'et}}}
\newcommand{\image}[1]{\mathrm{Im}(#1)}
\newcommand{\dsum}{{\oplus}}

\newcommand{\kernel}[1]{\mathrm{Ker}(#1)}
\newcommand{\cokernel}[1]{\mathrm{Cok}(#1)}

\newcommand{\cG}{\mathcal{G}}
\newcommand{\cO}{\mathcal{O}}
\newcommand{\cP}{\mathcal{P}}
\newcommand{\cR}{\mathcal{R}}

\newcommand{\cX}{\mathcal{X}}

\newcommand{\vA}{\mathbf{A}}
\newcommand{\vP}{\mathbf{P}}
\newcommand{\bA}{\mathbb{A}}

\newcommand{\bG}{\mathbb{G}}

\newcommand{\bP}{\mathbb{P}}
\newcommand{\bQ}{\mathbb{Q}}
\newcommand{\bZ}{\mathbb{Z}}

\renewcommand{\hom}[1]{\mathrm{Hom}_{#1}}
\usepackage[OT2,T1]{fontenc}
\DeclareSymbolFont{cyrletters}{OT2}{wncyr}{m}{n}
\DeclareMathSymbol{\Sha}{\mathalpha}{cyrletters}{"58}

\newcommand{\PTcouple}[1]{\langle#1\rangle_{\mathrm{PT}}}
\newcommand{\BMcouple}[1]{\langle#1\rangle_{\mathrm{BM}}}
\newcommand{\ABcouple}[1]{\langle#1\rangle_{\mathrm{ab}}}

\begin{document}

\makeatletter
\makeatother

	\author{Hui Zhang}
	\title[\bf{Abelianized Descent Obstruction}]{\bf{{Abelianized Descent Obstruction for 0-Cycles}}}
	
	\address{Hui ZHANG
	\newline University of Science and Technology of China,
	\newline School of Mathematical Sciences,
	\newline 96 Jinzhai Road,
	\newline 230026 Hefei, Anhui, China}
	\email{zhero@mail.ustc.edu.cn}
	
	\keywords{Brauer--Manin obstruction to Hasse principle, 0-cycles, descent obstruction}
	
	
	\setcounter{secnumdepth}{4}
	
	\maketitle 
	
\begin{abstract}
	
	Classical descent theory of \CT\ and Sansuc for rational points tells that, over a smooth variety $X$, the algebraic \BM\ subset equals the descent obstruction subset defined by a universal torsor. Moreover, Harari shows that the \BM\ subset equals the descent obstruction subset defined by torsors under connected linear groups. 
	
	By using the abelian cohomology theory by Borovoi, we define abelianized descent obstructions for 0-cycles by torsors under connected linear groups. As an analogy, we show the equality between the \BM\ obstruction and the abelianized descent obstruction for 0-cycles. We also show that the abelianized descent obstruction is the closure of the descent obstruction defined by Balestrieri and Berg when $X$ is a projective rationally connected variety or a projective K3 surface.
	
\end{abstract}

\section{Introduction}

	Let $k$ be a number field, $\Omega_k$ the set of places of $k$ and $\vA$ the adelic ring of $k$. Fix an algebraic closure $\overline k$ of $k$. Let $X$ be a variety over $k$. Set  $\br(X)=\mathrm{H}_{\et}^2(X,\bG_m)$, $\br_0(X)=\image{\br(k)\rightarrow\br(X)}$, $\br_1(X)=\kernel{\br(X)\rightarrow\br(X_{\overline k})}$ and $\br_a(X)=\br_1(X)/\br_0(X)$. Given a place $v\in\Omega_k$, denote $X_v=X_{k_v}$. Let $\zcycle{}(X)$ be the group of 0-cycles on $X$.
	
	According to the Brauer--Manin pairing 
	\begin{equation*}
		\BMcouple{\cdot,\cdot}:\brgp{X}\times X(\vA)\rightarrow\bQ/\bZ,
	\end{equation*}given a subset $B$ of $\brgp{X}$, there is a subset $X(\vA)^B=\{(P_v)\in X(\vA):\BMcouple{\alpha,(P_v)_{v\in\Omega_k}}=0$ for each $\alpha\in B\}$ of $X(\vA)$. By Brauer--Hasse--Noether exact sequecne, one has the inclustions $X(k)\subseteq X(\vA)^B\subseteq X(\vA)$.

	Following Manin, \CT\ introduced a similar pairing \begin{equation*}
		\BMcouple{\cdot,\cdot}:\brgp{X}\times \zcyclef{}{\vA}(X)\rightarrow\bQ/\bZ,
	\end{equation*}in his article \cite[]{CT95}. Here $\zcyclef{}{\vA}(X)=\coprod\limits_{\delta\in\bZ}\Big(\prod\limits_{v\in\Omega_k}\zcycle{\delta}(X_v)\Big)$ if $X$ is proper. When $X$ is not proper, see Section 2 for the precise definition. Given a subset $B$ of $\brgp X$, we define a subgroup $\zcyclef{}{\vA}(X)^B=$\{$(x_v)\in\zcyclef{}{\vA}(X):\BMcouple{\alpha,(x_v)}=0$ for each $\alpha\in B$\}. Still, by Brauer--Hasse--Noether exact sequecne, one can prove that $\zcycle{}(X)\subseteq\zcyclef{}{\vA}(X)^B\subseteq\zcyclef{}{\vA}(X)$. When $B=\br(X)$ or $\br_1(X)$, we simply write this \BM\ subgroup as $\zcyclef{}{\vA}(X)^\br$ or $\zcyclef{}{\vA}(X)^{\br_1}$. Given an integer $r$, subset $\zcyclef{r}{\vA}(X)^{B}$ of $\zcyclef{}{\vA}(X)^{B}$ collects those adelic 0-cycles of degree $r$ at each place $v\in\Omega_k$.
	
	It is natural to ask whether it is possible to generalise these results for rational points to those for 0-cycles. Especially, we care whether there is a relation between the descent obstruction and the Brauer--Manin obstruction for 0-cycles.

	
	\CT\ and Sansuc established the descent theory for tori for rational points and defined universal torsors in \cite{CTS87}. Given a torus $G$ and a torsor $f:Y\rightarrow X$ under $G$, they define a descent subset $X(\vA)^f$ of $X(\vA)$ associated to $f$, that is \begin{equation}\label{def-f-desc}
	X(\vA)^f=\{(P_v)_{v\in\Omega_k}:f(P_v)\text{ lies in the image of }\cohomorf{1}{k,G}\rightarrow\prod\limits_{v\in\Omega_k}\cohomorf{1}{k_v,G}\}.
	\end{equation}Later, this descent theory was extended by Skorobogatov to groups of multiplicative type in \cite{Sko99}. The descent obstruction for universal torsors can be compared to the algebraic \BM\ obstruction in the following sense.
	
	\begin{theorem}[\cite{Sko01}, Theorem 6.1.2 (a) and Corollary 6.1.3 (1)]\label{minor1}
		Let $X$ be a smooth, proper and geometrically integral variety over a number field $k$ such that $\pic{X_{\overline k}}$ is of finite type. 
		
	(1) Let $f:Y\rightarrow X$ be a universal torsor over $X$, then we have \begin{equation*}
	X(\vA)^{\br_1}=X(\vA)^f.
	\end{equation*}
	
	(2) The algebraic \BM\ subset $X(\vA)^{\br_1}\neq\emptyset$ if and only if there is a universal torsor $f:Y\rightarrow X$ such that $Y(\vA)\neq\emptyset$.
	\end{theorem}

	Notice that a group of multiplicative type is always commutative. In this case, one can easily define a descent obstruction for 0-cycles with the help of corestriction homomorphisms. 
	
	Let $G$ be a commutative $k$-group. For each place $v\in\Omega_k$, one can define an extended evaluation paring
	\begin{equation}\label{pair1}
	(\cdot,\cdot)^v:\cohomorf1{X,G}\times\zcycle{}(X_v)\rightarrow\cohomorf{1}{k_v,G}
	\end{equation}by sending an element $f\in\cohomorf{1}{X,G}$ and a 0-cycle $x_v=\sum n_{v,i}.P_{v,i}$ to \begin{equation}\label{formula}
	(f,x_v)^v=\sum n_i.\cor{k(P_{v,i})/k}(f(P_{v,i}))
	\end{equation}where $k(P_{v,i})$ is the residue field of $X_v$ at $P_{v,i}$ and $f(P_{v,i})$ is the (canonical) evaluation of $f$ at $P_{v,i}$. This paring defines a subset $\zcyclef{}{\vA}(X)^f=$\{$(x_v)_{v\in\Omega_k}\in\zcyclef{}{\vA}(X):((f,x_v)^v)_{v\in\Omega_k}$ lies in the image of $\cohomorf{1}{k,G}\rightarrow\prod\limits_{v\in\Omega_k}\cohomorf{1}{k_v,G}$\} of $\zcyclef{}{\vA}(X)$ for each torsor $f\in\cohomorf1{X,G}$.

	Then one can obtain the following analogy of Theorem \ref{minor1}.


	
	\begin{theorem}[Theorem \ref{cor1}]\label{cor1-copy}
		Let $X$ be a smooth, proper and geometrically integral variety over a number field $k$ such that $\pic{X_{\overline k}}$ is of finite type. 
		
	(1) Let $f:Y\rightarrow X$ be a universal torsor over $X$, then we have  \begin{equation*}
		\zcyclef{}{\vA}(X)^{\br_1}=\zcyclef{}{\vA}(X)^f.
	\end{equation*} 
	
	(2) If $\zcyclef{1}{\vA}(X)^{\br_1}\neq\emptyset$, then universal $X$-torsors exist.
		\end{theorem}

	

	The descent theory was also extended to a general algebraic group by Harari in \cite{Har02}. Harari defined the connected descent obstruction $X(\vA)^{\text{conn}}$ and compared it with the \BM\ obstruction in the following sense.


\begin{theorem}[\cite{Har02}, Theorem 2 and Remark 4]\label{minor2}
	Let $X$ be a smooth and geometrically integral variety over a number field $k$, then  \begin{equation*}
	X(\vA)^\br=X(\vA)^{\mathrm{conn}}.
	\end{equation*}
	\end{theorem}

	When $G$ is not commutative, there are no corestriction homomorphisms in the formula \reff{formula}. Hence, the definition of $\zcyclef{}{\vA}(X)^f$ does not make sense anymore in such case.

	To recover this, we require the technique of abelian Galois cohomology. The idea is as the following. By \cite{Bo89} or \cite{Bo98}, there is a functor $\cohomorab 1{K,G}$ from the category of connected linear groups over a field $K$ of characteristic 0 to the category of abelian groups. Thanks to the abelian Galois cohomology, we can define a \textit{abelianized} evaluation  \begin{equation*}
	(\cdot,\cdot)^v_\ab:\cohomorf1{X,G}\times\zcycle{}(X_v)\rightarrow\cohomorab{1}{k_v,G},
	\end{equation*}see Section 4 for details. Given a torsor $f\in\cohomorf1{X,G}$, we can also define the \textit{abelianized} descent subgroup $\zcyclef{}{\vA}(X)^f_\ab=$\{$(x_v)_{v\in\Omega_k}\in\zcyclef{}{\vA}(X):((f,x_v)_\ab^v)_{v\in\Omega_k}$ lies in the image of $\cohomorab{1}{k,G}\rightarrow\prod\limits_{v\in\Omega_k}\cohomorab{1}{k_v,G}$\}. Similarly, we define $\zcyclef{}{\vA}(X)^{\mathrm{conn}}_\ab=\bigcap\limits_{f:Y\xrightarrow GX}\zcyclef{}{\vA}(X)^f_\ab$ where the intersection is taken over all $X$-torsors $f:Y\xrightarrow{G}X$ under all connected linear groups $G$ over $k$. Then one has the following analogy as the case of rational points.

%


	\begin{theorem}[Theorem \ref{main3}]\label{main3-copy}
	Let $X$ be a smooth and geometrically integral variety over a number field $k$, then  \begin{equation*}
	\zcyclef{}{\vA}(X)^\br=\zcyclef{}{\vA}(X)^{\mathrm{conn}}_\ab.
	\end{equation*}
	\end{theorem}
	
	Theorem \ref{cor1-copy} is an analogy of Theorem \ref{minor1} while Theorem \ref{main3-copy} is an anology of Theorem \ref{minor2}. These establish the relation between the descent obstruction and the \BM\ obstruction for 0-cycles.

	In \cite[Definition 3.3]{BB24}, Balestrieri and Berg also introduced an $f$-descent obstruction $\zcyclef{}{\vA}(X)^f_\BB$ for 0-cycles for $G$ non-necessarily connected, see also Definition \ref{def-desc-1}. Their insight comes from the $f$-descent decomposition \begin{equation}\label{g-inter}
	X(\vA)^f=\bigcup\limits_{\sigma\in\cohomorf{1}{k,G}}f^\sigma(Y^\sigma(\vA))
	\end{equation}for adelic points. For 0-cycles, the equation (\ref{def-f-desc}) and the equation (\ref{g-inter}) do not clearly coincide. However, if we give $\zcyclef{}{\vA}(X)$ a topology that is natural to describe the weak approximation for 0-cycles, then we will see that the \BM\ subgroup $\zcyclef{}{\vA}(X)^\br$ and the abelianized $f$-descent subgroup $\zcyclef{}{\vA}(X)^f_\ab$ are closed when $f$ is a torsor under a connected linear group $G$. It is not clear whether the $f$-descent subgroup $\zcyclef{}{\vA}(X)_\BB^f$ is closed, but we are able to provide the following relation between this descent subgroup and our abelianized $f$-descent subgroup under suitable assumptions on $X$. Here is the theorem for this result.

	\begin{theorem}[Corollary \ref{main-geo}]\label{main-geo-copy} Let $X$ be a smooth, projective and geometrically integral variety over a number field $k$, $G$ a connected linear $k$-group and $f:Y\rightarrow X$ a torsor under $G$. In each of the following cases:
	
	(i) $X$ is rationally connected,
	
	(ii) $X$ is a K3 surface,
	
	we obtain that 	
	\begin{equation*}
	\zcyclef{}\vA(X)_\ab^f=\overline{\zcyclef{}\vA(X)_{\mathrm{BB}}^f}.
	\end{equation*}
\end{theorem}

%
%
%
%
	

	This paper is organized as the following. In Section 2, we define the descent obstruction for 0-cycles with respect to a commutative group and then prove Theorem \ref{cor1-copy}. In Section 3,  we recall some facts of abelian Galois cohomologies and develop some tools  required for the proof of Theorem \ref{main3-copy}. In Section 4, we define the abelianized descent obstruction for 0-cycles with respect to a connected linear group and then prove Theorem \ref{main3-copy}. Last in Section 5, we define a topology on $\zcyclef{}\vA(X)$ and prove Theorem \ref{main-geo-copy}. At the end of Section 5, we also give an application of Theorem \ref{cor1-copy}, Theorem \ref{main3-copy} and Theorem \ref{main-geo-copy}.
	
%
%
%
%
%
%
%
%


\section{Descent obstruction for a group of multiplicative type}

\subsection{Notations}\,

	This article only deals with $\et$ale cohomology, therefore, we omit the subscript $\et$ in the $\et$ale cohomology $\cohomoret{\ast}{\cdot,\cdot}$. 

	Let $k$ be a field of characteristic 0, $\Gamma_k$ denotes the absolute Galois group $\text{Gal}(\overline k/k)$ of $k$ and for a finite extension $K$ of $k$, $\Gamma_{K/k}$ denotes the Galois group $\text{Gal}(K/k)$. 
	
	A variety $X$ over $k$ is defined to be a separated scheme of finite type over $k$. Let $x=\sum n_i.P_i$ ($n_i\not=0$) be a 0-cycle on $X$, we call \{$P_i$\} the support of $x$ and call $\sum n_i.[k(P_i):k]$ the degree of $x$.
	
	Let $f:Y\rightarrow X$ be a morphism of $k$-varieties, then there is a pushforward homomorphism $f_*:\zcycle{}(Y)\rightarrow\zcycle{}(X)$ by sending $\sum n_i.Q_i$ to $\sum n_i.[k(Q_i):k(f(Q_i))].f(Q_i)$. Pushforward homomorphism does not change the degree of 0-cycles.
	
	Let  $G$ be a commutative $k$-group, for each integer $i$, we can define a pairing
	\begin{equation*}
	(\cdot,\cdot):\cohomorf{i}{X,G}\times\zcycle{}(X)\rightarrow\cohomorf{i}{k,G}
	\end{equation*}by the rule
	 $(f,x)=\sum n_i.\cor{k(P_i)/k}(f(P_i))$ where $x=\sum n_i.P_i$ is the formal sum of finitely many closed points $\{P_i\}$ on $X$, $k(P_i)$ is the residue field of $X$ at $P_i$, $f(P_i)$ is the restriction of $f$ in $\cohomorf{i}{k(P_i),G}$ and $\cor{k(P_i)/k}$ is the corestriction homomorphism $\cohomorf{i}{k(P_i),G}\rightarrow\cohomorf{i}{k,G}$.

	Now, let $k$ be a number field. For a place $v\in\Omega_k$ and $L$ a finite extension of $k_v$, let $|\cdot|_L$ be the valuation of $L$. Denote $\cO_L=$\{$a\in L:|a|_L<1$\} be the discrete valuation ring associated to $L$. We simply denote $\cO_v=\cO_{k_v}$. Let $S$ be a finite subset of $\Omega_k$ containing all archimedean places, then $\cO_{S}=$\{$a\in k:v(a)\geq0$ for each $v\not\in S$\} is the ring of $S$-integers of $k$.
	
	Assume that $X$ is smooth and geometrically integral over $k$. 
	
	Let $\delta$ be an integer. We denote by $\zcyclef{\delta}{\vA}(X)$ the subset of $\prod\limits_{v\in\Omega_k}\zcycle{\delta}(X_v)$ collecting those $(x_v)_{v\in\Omega_k}$ with the property that there is a finite subset $S\subseteq\Omega_k$ containing all archimedean places of $k$, a model $\cX_S$ of $X$ over $\cO_S$ such that for each $v\not\in S$, each closed point $P_v$ in the suppport of $x_v$ comes from $\cX_S(\cO_{k(P_v)})$ where $\cO_{k(P_v)}$ is the integral closure of $\cO_v$ in $k(P_v)$. When $X$ is proper, we have that $\zcyclef{\delta}{\vA}(X)=\prodf{v\in\Omega_k}\zcycle{\delta}(X_v)$. Let $\zcyclef{}{\vA}(X)=\coprod\limits_{\delta\in\bZ}\zcyclef{\delta}{\vA}(X)$. We call an element $(x_v)_{v\in\Omega_k}$ in $\zcyclef{}{\vA}(X)$ (resp. in $\zcyclef{\delta}{\vA}(X)$) an adelic 0-cycle (resp. adelic 0-cycle of degree $\delta$) on $X$.
	
	Let $f:Y\rightarrow X$ be a morphism of $k$-varieties. Pushforward homomorphisms $\zcycle{}(Y_v)\rightarrow\zcycle{}(X_v)$ at each place $v\in\Omega_k$ induce a pushforward $\zcyclef{}{\vA}(Y)\rightarrow\zcyclef{}{\vA}(X)$ on the level of adelic 0-cycles preserving their degrees.
	
	Recall that the \BM\ pairing for 0-cycles 
	\begin{equation} \label{pairing}
		\BMcouple{\cdot,\cdot}:\brgp{X}\times\zcyclef{}{\vA}(X)\rightarrow\mathbb{Q}/\mathbb{Z}
	\end{equation}sends an element $\alpha\in\br(X)$ and an adelic 0-cycle $(x_v)_{v\in\Omega_k}$ to $\sumf{v\in\Omega_k}\inv_v(\alpha,x_v)$. This sum is a finite sum immediately due to the definition of $\zcyclef{}{\vA}(X)$. Hence, it is well-defined. 
	
	By Brauer--Hasse--Noether exact sequecne with a restriction-corestriction argument, we know that the group $\br_0(X)$ is orthogonal to $\zcyclef{}{\vA}(X)$.
	
	
	If $f$ is an element of $\cohomorf{1}{X,G}$, set $\zcyclef{}{\vA}(X)^f=\{(x_v)_{v\in\Omega_k}\in\zcyclef{}{\vA}(X):((f,x_v))_{v\in\Omega_k}$ lies in the image of $\cohomorf{1}{k,G}\rightarrow\prod_{v\in\Omega_k}\cohomorf{1}{k_v,G}\}$. 
	
	
\subsection{Restricted product $\vP^1(k,G)$}\,

	Let $G$ be a group of finite type over  a number field $k$ (it is smooth). Let $X$ be a smooth and geometrically integral variety over $k$ and $f\in\cohomorf{1}{X,G}$ an element.
	
	There is a finite subset $S$ of $\Omega_k$ such that the following assumptions are satisfied.
	
	$\bullet$ $S$ contains all archimedean places of $k$.
	
	$\bullet$ There is a smooth $U$-scheme $\cX$ with geometrical integral fibres such that generic fibre of $\cX$ is $X$. Here $U=\spec{\cO_S}$ is a non-empty open subset of $\spec{\cO_k}$.
	
	$\bullet$ There is a smooth group $U$-scheme $\cG$ such that generic fibre $\cG$ is $G$. 
	

	$\bullet$ The element $f$ comes from $\cohomorf{1}{\cX,\cG}$.
	
	Then we have the following basic facts.
	
	\begin{lemma}\label{loc:tv1}
	Notations as above, take a place $v\not\in S$ and a closed point $P_v$ on $X_v$ coming from $\cX(\cO_{k(P_v)})$, then $f(P_v)\in\cohomorf1{k(P_v),G}$ lies in the image of restriction map \begin{equation*}
	\cohomorf1{\cO_{k(P_v)},\cG}\rightarrow\cohomorf1{k(P_v),G}.
	\end{equation*}
	\end{lemma}
	
	\begin{proof}
	
	Let	$P_v$  comes from a unique integral point $\cP_v\in\cX(\cO_{k(P_v)})$. Then the lemma follows from the commutativity of the diagram
		\begin{equation*}
		\xymatrix{
		\cohomorf1{\cX,\cG}\ar[r]^{(\cP_v)\quad}\ar[d]&\cohomorf{1}{\cO_{k(P_v)},\cG}\ar[d]\\
		\cohomorf1{X,G}\ar[r]_{(P_v)\quad}&\cohomorf{1}{k(P_v),G}}.
		\end{equation*}
	\end{proof}

	Let $\vP^1(k,G)$ denote the restricted product $\prod'\cohomorf{1}{k_v,G}$ with respect to images of $\cohomorf{1}{\cO_v,\cG}$ in $\cohomorf{1}{k_v,G}$ for each finite place $v\in S$, then $\vP^1(k,G)$ is independent on the choice of $S$ or $\cG$ by a limit argument. With the Lemma \ref{loc:tv1}, we immediately obtain the following corollary.
	
	\begin{corollary}\label{corr1}
	If $X$ is smooth and geometrically integral over a number field $k$, $G$ is a commutative $k$-group of finite type, take an element  $f\in\cohomorf{1}{X,G}$ and an adelic 0-cycle $(x_v)_{v\in\Omega_k}\in\zcyclef{}{\vA}(X)$, then $((f,x_v))_{v\in\Omega_k}$ lies in $\vP^1(k,G)$.
	\end{corollary}

	\begin{proof}Notations as above. We may assume that $S$ satisfies the property for $(x_v)_{v\in\Omega_k}$ (in its definition) holds by enlargeing $S$.
	Since $G$ is commutative, we may also assume $\cG$ is commutative over $\spec{\cO_S}$ by enlarging $S$. Since for a place $v\not\in S$ and a closed point $P_v$ in the support of $x_v$, the following diagram\begin{equation*}
	\xymatrix{\cO_{k(P_v)}\ar[r]&k(P_v)\\
	\cO_v\ar[u]\ar[r]&k_v\ar[u]}
	\end{equation*}is cocartesian. It induces a commutative diagram 
	\begin{equation*}
	\xymatrix{
	\cohomorf{1}{\cO_{k(P_v)},\cG}\ar[r]\ar[d]_{\cor{\cO_{k(P_v)}/\cO_v}}&\cohomorf{1}{k(P_v),G}\ar[d]^{\cor{k(P_v)/k_v}}\\
	\cohomorf{1}{\cO_v,\cG}\ar[r]&\cohomorf{1}{k_v,G}}.
	\end{equation*}Now the corollary is clear with the help of Lemma \ref{loc:tv1}.
	\end{proof}

	For a connected linear group $G$, we have the following lemma, which will be used in Section 4.

	\begin{lemma}\label{loc:tv2}
	Let $X$ be a smooth and geometrically integral over a number field $k$, $G$ a connected linear $k$-group. Let $f\in\cohomorf{1}{X,G}$ be a torsor. Let $(x_v)_{v\in\Omega_k}\in\zcyclef{}{\vA}(X)$ be an adelic 0-cycle, then there is a finite subset $S\subseteq\Omega_k$ such that for any $v\not\in S$ and any closed point $P_v$ lying in the support of $x_v$, the evaluation $f(P_v)\in\cohomorf{1}{k(P_v),G}$ is trivial.
	
	\end{lemma}
	\begin{proof}Take $S,\cX$ and $\cG$ as above. By enlarging $S$, we may assume that $S$ satisfies the property of $(x_v)_{v\in\Omega_k}$ (in its definition), $\cG$ is smooth affine  over $U$ and all fibres of $\cG$ over $U$ are connected. Since $\cohomorf{1}{\cO_{k(P_v)},\cG}=1$ by Lang's theorem and Hensel's lemma, moreover, $f(P_v)$ comes from $\cohomorf{1}{\cO_{k(P_v)},\cG}$ by Lemma \ref{loc:tv1}, the lemma follows.
	\end{proof}			
	

\subsection{Comparison of the algebraic \BM\ obstruction and the descent obstruction for universal torsors}\,

	By spectral sequence $\cohomorf p{k,\cohomorf q{X_{\overline k},\bG_m}}\Rightarrow\cohomorf{p+q}{X,\bG_m}$, there is an exact sequence \begin{equation*}
		0\rightarrow\brgpf0{X}\rightarrow\brgpu X\xrightarrow\rho\cohomorf{1}{k,\pic{X_{\overline k}}}\rightarrow\cohomorf{3}{k,\bG_m}=0.
		\end{equation*}Given a $\Gamma_k$-module $M$ and a homomorphism $\lambda:M\rightarrow\pic{X_{\overline k}}$ of $\Gamma_k$-modules, there is a homomorphism $\lambda_*:\cohomorf1{k,M}\rightarrow\cohomorf1{k,\pic{X_{\overline k}}}$ and a subgroup $\br_\lambda=\rho^{-1}(\lambda_*(\cohomorf1{k,M}))$ of $\brgpu X$.
		
	Let $G$ be a $k$-group of multiplicative type. By spectral sequence $\mathrm{Ext}_k^p(\hat G,R^p\pi_*\bG_m)\Rightarrow\mathrm{Ext}_X^{p+q}(\pi^*\hat G,\bG_m)$ where $\pi:X\rightarrow\spec{k}$ is the canonical map, one obtains the following fundamental exact sequence (see \cite[Thm 1.5.1]{CTS87} or \cite[Thm 2.3.6]{Sko01}) \begin{equation}\label{chi_seq}
	0\rightarrow\cohomorf1{k,G}\rightarrow\cohomorf1{X,G}\xrightarrow\chi\textrm{Hom}_{\text{grp}}(\hat G,\pic{X_{\overline k}})\xrightarrow\partial\cohomorf2{k,G}\rightarrow\cohomorf{2}{X,G}.
	\end{equation}An element $f$ of $\cohomorf1{X,G}$ corresponds a homomorphism $\lambda=\chi(f):\hat G\rightarrow\pic{X_{\overline k}}$ and we call this $\lambda$ the type of $f$. Take $M=\hat G$, one also has a subgroup $\br_\lambda=\br_{\chi(f)}$ of $\br_1(X)$ determined by $f$. 
	
	
	The following proof is an analogy of \cite[Proposition 3.1]{HS02} for 0-cycles.
	
		\begin{theorem}\label{main1}
		Let $k$ be a number field, $G$ a group of multiplicative type over $k$, and $X$ a smooth and geometrically integral variety over $k$. Let $f\in\cohomorf1{X,G}$, then one has  \begin{equation*}
		\zcyclef{}{\vA}(X)^{\br_{\chi(f)}}=\zcyclef{}{\vA}(X)^f.
		\end{equation*}
		\end{theorem}
	
	\begin{proof}
	Let $(x_v)_{v\in\Omega_k}\in\zcyclef{}{\vA}(X)$ be an adelic 0-cycle. Notice that $((f,x_v))_{v\in\Omega_k}$ lies in $\vP^1(k,G)$ by Corollary \ref{corr1}. By \cite[Thm 6.3]{Der11}, the adelic 0-cycle $(x_v)_{v\in\Omega_k}$ lies in $\zcyclef{}{\vA}(X)^f$ if and only if $\sum\limits_{v\in\Omega_k}\inv_v(\alpha\cup(f,x_v))=\BMcouple{\alpha\cup f,(x_v)_{v\in\Omega_k}}=0$ for each $\alpha\in\cohomorf{1}{k,\hat G}$. By \cite[Thm 1.4]{HS02}, for any element $\alpha\in\cohomorf1{k,\hat G}$, we have $\rho(\alpha\cup f)=\lambda_*(\alpha)$. Hence, any element $\beta\in\br_{\chi(f)}$ can be written as form $(\alpha\cup f)+\beta_0$ where $\beta_0\in\brgpf0 X$. Therefore, $(x_v)_{v\in\Omega_k}$ lies in $\zcyclef{}{\vA}(X)^f$ if and only if $(x_v)_{v\in\Omega_k}$ lies in $\zcyclef{}{\vA}(X)^{\br_{\chi(f)}}$. 
	\end{proof}

To complete the proof of Theorem \ref{cor1}, we also need the following lemma.

\begin{lemma}\label{06-14}
Let $X$ be a smooth, proper and geometrically integral variety over a number field $k$. Let $M$ be a $\Gamma_k$-module of finite type and $\lambda:M\rightarrow\pic{X_{\overline k}}$ a homomorphsim of $\Gamma_k$-modules. If $\zcyclef{1}{\vA}(X)^{\br_\lambda}\not=\emptyset$, then there is a torsor $f\in\cohomorf{1}{X,G}$ such that the type of $f$ is $\lambda$.
\end{lemma}

\begin{proof}
Due to the exact sequence \reff{chi_seq}, we only need to show that $\partial(\lambda)=0$. Let $G$ be a $k$-group of multiplicative type such that $\hat G=M$. Recall that in paper \cite[Theorem 5.7]{Der11} or \cite[Section 3]{ZH}, one may define a perfect Poitou--Tate pairing \begin{equation*}
\PTcouple{\cdot,\cdot}:\Sha^2(G)\times\Sha^1(\hat G)\rightarrow\bQ/\bZ.
\end{equation*}In \cite[Section 3]{ZH}, one also has an invariant map $\inv:\Sha^1(\Delta)\rightarrow\bQ/\bZ$. Here $\Delta$ is the mapping cone of the canonical morphism \begin{equation*}
		\bG_{m,k}[1]\rightarrow(\tau_{\leq1}R\pi_*(\bG_{m,X}))[1].
	\end{equation*}Since $\overline k[X]^*=\overline k^*$, the complex $\Delta$ is isomorphic to $\pic{X_{\overline k}}[0]$ in the derived category $\mathrm{D}(k)$. By \cite[Lemma 3.7]{ZH}, they are connected by a formula that $\PTcouple{\partial(\lambda),A}=\inv(\lambda_*(A))$ for any $A\in\Sha^1(\hat G)$ where $\lambda_*$ is the induced homomorphism $\cohomorf{1}{k,\hat G}=\cohomorf{1}{k,M}\rightarrow\cohomorf{1}{k,\pic{X_{\overline k}}}$. Take an adelic 0-cycle $(x_v)_{v\in\Omega_k}\in\zcyclef{1}\vA(X)^{\br_\lambda}$ and take a pre-image $\alpha\in\br_1(X)$ of $\lambda_*(A)\in\Sha^1(\Delta)\subseteq\brgpa{X}$, then one obtains that $\inv(\lambda_*(A))=\sum\limits_{v\in\Omega_k}\inv_v(s_v(\alpha(X_v)))=0$ by snake lemma with the fact that $\alpha$ lies in $\br_\lambda$. Here $s_v$ is the retraction $\br_1(X_v)\rightarrow\br(k_v)$ defined by the 0-cycle $x_v$ of degree 1. 
	
	A detailed proof can also be found in \cite[Proposition 3.8]{ZH} by slight modifications.  
\end{proof}

\begin{definition}
Let $X$ be a smooth, proper and geometrically integral variety over a number field $k$ such that $\pic{X_{\overline k}}$ is a $\Gamma_k$-module of finite type. Let $G$ be a $k$-group of multiplicative type such that $\hat G\simeq\pic{X_{\overline k}}$, then a universal $X$-torsor is a  torsor $f\in\cohomorf1{X,G}$ such that its type $\chi(f)$ is the canonical isomorphism $\hat G\simeq\pic{X_{\overline k}}$. Notice that the universal $X$-torsor is unique up to a twist of $\sigma\in\cohomorf{1}{k,G}$ due to the excat sequence \reff{chi_seq}.
\end{definition}

\begin{theorem}\label{cor1}
		Let $X$ be a smooth, proper and geometrically integral variety over a number field $k$ such that $\pic{X_{\overline k}}$ is of finite type. 
		
	(1) Let $f:Y\rightarrow X$ be a universal torsor over $X$, then we have  \begin{equation*}
			\zcyclef{}{\vA}(X)^{\br_1}=\zcyclef{}{\vA}(X)^f.
			\end{equation*} 
	
	(2) If $\zcyclef{1}{\vA}(X)^{\br_1}\neq\emptyset$, then universal $X$-torsors exist.
		\end{theorem}
		
\begin{proof}
The part (1) is no other than Theorem \ref{main1} and the part (2) is immediately the corollary of Lemma \ref{06-14}.
\end{proof}

%
%
%
%
%
	
\section{abelianization of Galois cohomology of connected reductive group}
	In this section, we recall basic concepts of abelian Galois cohomology, see \cite{Ko86}, \cite{Bo89} or \cite{Bo98} for a referrence. Throughout this section, all fields are of characteristic 0 and all algebraic groups over some field are linear.

\subsection{Abelian Galois cohomology of reductive groups}\,

	Let $G$ be a connected reductive group over a field $K$ (it will be taken to be a number field or a local field later). Let $G^{ss}$ denote the derived group of $G$ (it is semi-simple) and let $G^{sc}$ be the universal covering group of $G^{ss}$ (it is simply connected). Let $G^{tor}=G/G^{ss}$ be the maximal quotient torus of $G$. Let $\psi$ denote the composition of morphisms\begin{equation}\label{15}
			G^{sc}\rightarrow G^{ss}\hookrightarrow G.
			\end{equation}and denote $G^\bullet=(G^{sc}\xrightarrow{\psi}G)$.
			
	\begin{definition}\label{cdef}
	Let $Z$ be the center of $G$ and let $Z^{(sc)}=\psi^{-1}(Z)$ be the center of $G^{sc}$. Denote $Z^\bullet=(Z^{(sc)}\rightarrow Z)$. It is a complex of algebraic $K$-groups, putting $Z$ in degree 0 and $Z^{(sc)}$ in degree $-1$. We identify $Z$ (and other commutative algebraic $K$-groups) with the $\text{Gal}(\overline K/K)$-module associated to it canonically. Define $\cohomorab{i}{K,G}=\cohomorh{i}{K,Z^\bullet}$ for $i\geq-1$ where $\cohomorh{i}{K,\ast}$ is the Galois hyper-cohomology.
	\end{definition}

	Let $T\subseteq G$ be a maximal torus, then $T^{(sc)}=\psi^{-1}(T)\rightarrow T$ is complex of $K$-tori, putting $T$ in degree 0 and $T^{(sc)}$ in degree $-1$. Denote $T^\bullet=(T^{(sc)}\rightarrow T)$.
	
	\begin{lemma}[\cite{Bo98}, Lemma  2.4.1]\label{qs1}
	The canonical injection 
	\begin{equation*}
	Z^\bullet=(Z^{(sc)}\xrightarrow{\psi_Z} Z)\rightarrow (T^{(sc)}\xrightarrow{\psi_T} T)=T^\bullet
	\end{equation*}is a quasi-isomorphism, i.e. induced morphisms $\kernel{\psi_Z}\rightarrow\kernel{\psi_T}$ and  $\cokernel{\psi_Z}\rightarrow\cokernel{\psi_T}$ are isomorphisms. Hence, this quasi-isomorphism induces a canonical isomorphisms $\cohomorab{i}{K,G}\simeq\cohomorh{i}{K,T^\bullet}$ of Galois hyper-cohomology for each $i\geq-1$.
	\end{lemma}
	
	Let $\varphi:G_1\rightarrow G_2$ be a morphism of connected reductive groups. Let $Z_j^\bullet$ be the complex associated to centers of $G_j^\bullet$ for $j=1,2$. Let $T_2$ be a maximal torus of $G_2$ containing $\varphi(Z_1)$, then one induces a diagram\begin{equation*}
	Z_1^\bullet\rightarrow T_2^\bullet\leftarrow Z_2^\bullet
	\end{equation*}and then a diagram\begin{equation*}
	\cohomorab{i}{K,G_1}\rightarrow\cohomorh{i}{K,T_2^\bullet}\simeq\cohomorab{i}{K,G_2}
	\end{equation*}by Lemma \ref{qs1}. Let $\varphi_*$ be the morphism $\cohomorab{i}{K,G_1}\rightarrow\cohomorab{i}{K,G_2}$ above. It is independent on the choice of $T_2$ by \cite[Proposition 2.6.5 and Proposition 2.8]{Bo98}. 
	
	Furthermore, let \begin{equation*}
	\xymatrix{1\ar[r]& G_1\ar[r]& G_2\ar[r]& G_3\ar[r]& 1}
	\end{equation*}be an exact sequence of  reductive $K$-groups, one induces a long hyper-cohomology exact sequence\begin{equation*}
		\xymatrix{\cdots\rightarrow\cohomorab{i}{K,G_1}\rightarrow\cohomorab{i}{K,G_2}\rightarrow\cohomorab{i}{K,G_3}\xrightarrow{\delta_i}\cohomorab{i+1}{K,G_1}\rightarrow\cdots}
		\end{equation*}by \cite[Proposition 2.10]{Bo98}. Hence, functors with the connecting homomorphisms $\{\cohomorab{i}{K,\bullet},\delta_i\}$ induce a cohomological $\delta$-functor from the category of connected reductive $K$-groups to the category of abelian groups. 
	
	\begin{remark}\label{torus} Here are some examples. 
	
	(1) If $G$ is a torus, then $\cohomorf{i}{K,G}\simeq\cohomorab{i}{K,G}$ for any $i$.

	(2) If $G$ is semi-simple, then $\cohomorab{i}{K,G}\simeq\cohomorf{i+1}{K,\kernel{\psi}}$ for any $i$.
	
	(3) If $G^{ss}$ is simply connected, then $\cohomorab{i}{K,G}\simeq\cohomorab{i}{K,G^{tor}}$ for any $i$.
	\end{remark}

\subsection{Abelianization map in degree 1}\,

	To construct the abelianization map $\ab^1$ in degree 1, we have at least two ways. One way needs crossed modules, introduced by J.H.C. Whitehead \cite{Wh49}, then maps $\ab^1$ are defined in terms of cocycles, see \cite[Sect 3]{Bo98} for details. The other way uses $z$-extensions of reductive groups, introduced by Langlands \cite{La79}, \cite{La89}, see also \cite{Ko86} or \cite{Bo89}. Both ways are essentially equivalent due to Theorem \ref{abthm1} below. We sketch the construction by $z$-extensions in the following.

	First, for any $K$-torus $T$, we have canonical isomorphisms \begin{equation}\label{ab1}
	\ab^i_T:\cohomorf{i}{K,T}\simeq\cohomorab{i}{K,T}
	\end{equation}by Remark \ref{torus}, (1). These isomorphisms form an isomorphism of cohomological $\delta$-functors $T\leadsto\cohomorf{i}{K,T}$ and $T\leadsto\cohomorab{i}{K,T}$ from the category of $K$-tori to the category of abelian groups. We omit subscript $T$ in $\ab^i_T$ if $T$ is clear.
	
	\begin{theorem}[\cite{Bo89}, Theorem 3.2]\label{abthm1}
	These isomorphisms $\ab^1$ above can be uniquely prolonged to morphisms of functors 
	\begin{equation*}
	\xymatrix{\ab^1=\ab^1_G:\cohomorf{1}{K,G}\ar[r]&\cohomorab{1}{K,G}}
	\end{equation*}from the category of reductive $K$-groups to the category of pointed sets. This also means each $\ab^1$ is a morphism of pointed sets.
	\end{theorem}
	
	We give the definition of maps $\ab^1$ in the following. After the definition is given, the proof of Theorem \ref{abthm1} will be complished.

	Now, we extend $\ab^1$ with the following definition of $z$-extensions.
	
	\begin{definition}
	Let $G$ be a connected reductive $K$-group. A central extension \begin{equation}\label{z-ext}
	1\rightarrow F\rightarrow H\xrightarrow{p} G\rightarrow 1
	\end{equation}of $G$ is called a $z$-extension if $H^{ss}$ is simply connected and $F$ is a product of tori of the form $\wres{L/K}(\bG_m)$ for finite extensions $L/K$.
	
	Furthermore, let $f\in\cohomorf{1}{K,G}$ be a torsor, then the $z$-extension \reff{z-ext} is called a $f$-lifting $z$-extension if $f$ comes from $\cohomorf{1}{K,H}$ along the canonical map $\cohomorf{1}{K,H}\xrightarrow{p_*}\cohomorf{1}{K,G}$.
	\end{definition}	

	The following two lemmas guarantee the existence of $f$-lifting $z$-extensions and lemma \ref{exist2} slightly generates \cite[Corollary 3.4.5]{Bo89} to fit our requirements.
	
	\begin{lemma}[\cite{Bo89}, Lemma 3.4.2]\label{exist1} Let $G$ be a connected reductive $K$-group and $\psi':G^{\text{sc}}\times Z^o\rightarrow G,(g,z)\mapsto\psi(g).z$ be the canonical covering of $G$ where $Z^o$ is the neutral connected component of the center $Z$ of $G$. Set $A=\kernel{\psi'}$ (it is a finite abelian group) and let $L/K$ be a finite Galois extension such that $\text{Gal}(\overline L/L)$ acts on the character group $\chi^*(A)$ trivially, then there is a $z$-extension \begin{equation*}
	1\rightarrow F\rightarrow H\rightarrow G\rightarrow 1
	\end{equation*}such that $F\simeq(\wres{L/K}(\bG_m))^n$
	
	\end{lemma}
	
	\begin{remark}
	The proof of Lemma \ref{exist1} is the same as the one given in \cite{Bo89}. Since \cite{Bo89} is not published, we put it here to convince readers.
	\end{remark}
	
	\begin{proof}
	Let $\Delta=\text{Gal}(L/K)$, then there is a surjective homomorphism $s:M\rightarrow\chi^*(A)$ of $\Delta$-moduls, where $M$ is a $\bZ[\Delta]$-free module. Let $F$ be a $K$-torus such that $\chi^*(F)=M$, then it is a torus of the form $(\wres{L/K}(\bG_m))^n$. Since $s$ is surjective, the induced homomorphism $s^*:A\rightarrow F$ is injective. We set $H=(G^{sc}\times Z^o\times F)/A$ and define $p:H\rightarrow G=(G^{sc}\times Z^o)/A$ to be the canonical projection induced from\begin{equation*}
	G^{sc}\times Z^o\times F\rightarrow G^{sc}\times Z^o.
	\end{equation*}Then $\kernel{p}=F$ and $H^{ss}\simeq G^{sc}$ (because $A\hookrightarrow F$ is injective). The lemma is proved.
	\end{proof}
	
	\begin{lemma}[\cite{Bo89}, Corollary 3.4.5 $+\epsilon$] \label{exist2} 
	Let $G$ be a connected reductive $K$-group, $L$ a finite extension of $K$, $f\in\cohomorf{1}{L,G}$ be an element, then there exists a $z$-extension $H\rightarrow G$ over $K$ such that $f$ comes from $\cohomorf{1}{L,H}$.  
	\end{lemma}
	
	\begin{proof}
	Take a finite extension $L'$ of $L$ such that $L'$ is a Galois extension of $K$, $f$ maps to $1\in\cohomorf{1}{L',G}$ and $L'$ splits $G$. Since $L'$ splits $G$, $L'/K$ satisfies the condition of Lemma \ref{exist1}. Hence, we apply Lemma \ref{exist1} to the extension $L'/K$, then there is  a $z$-extension \begin{equation*}
	\xymatrix{1\ar[r]&F\ar[r]&H\ar[r]&G\ar[r]&1}
	\end{equation*}with $F\simeq(\wres{L'/K}(\bG_m))^n$. Notice that $F_L\simeq(\wres{L'/L}(\bG_m))^{n.[L:K]}$. Consider the commutative diagram with exact rows \begin{equation*}
	\xymatrix{\cohomorf{1}{L,H}\ar[r]\ar[d]&\cohomorf{1}{L,G}\ar[d]\ar[r]&\cohomorf{2}{L,F}\ar[d]\\
	\cohomorf{1}{L',H}\ar[r]&\cohomorf{1}{L',G}\ar[r]&\cohomorf{2}{L',F}}.
	\end{equation*}By the choice of $Z$, the restriction homomorphism $\cohomorf{2}{L,F}\rightarrow\cohomorf{2}{L',F}$ is injective. Therefore, $f$ comes from $\cohomorf{1}{L,H}$ by a diagram chasing.
	\end{proof}
	
	Here is the construction of $\ab^1$ by $z$-extensions. 
	
	For  sake of language, let $\cR$ be the category of connected reductive $K$-groups and let $\cR^o$ be the full sub-category of $\cR$ collecting those $K$-groups $G$ such that $G^{ss}$ is simply connected.
	
	Firstly, we extend $\ab^1$ from the category of $K$-tori to $\cR^o$. For a reductive group $G\in\cR^o$, the diagram \begin{equation*}
	\xymatrix{\cohomorf{1}{K,G}\ar[r]\ar @{-->}[d]_{\ab^1_G}&\cohomorf{1}{K,G^{tor}}\ar[d]^{\ab^1_{G^{tor}}}_\sim\\
	\cohomorab{1}{K,G}\ar[r]^\sim&\cohomorab{1}{K,G^{tor}}}
	\end{equation*}forces us to define $\ab^1_G$ to be the composition\begin{equation*}
	\cohomorf{1}{K,G}\rightarrow\cohomorf{1}{K,G^{tor}}\simeq\cohomorab{1}{K,G}.
	\end{equation*}This clearly extends $\ab^1$ to $\cR^o$.
	
	For a reductive group $G\in\cR$, given an element $f\in\cohomorf{1}{K,G}$, take a $f$-lifting $z$-extension \begin{equation*}
	1\rightarrow F\rightarrow H\xrightarrow pG\rightarrow1.
	\end{equation*}Then we have a commutative diagram\begin{equation*}
	\xymatrix{\cohomorf{1}{K,F}\ar[r]\ar[d]^{\ab^1_F}_\sim&\cohomorf{1}{K,H}\ar[r]^{p_*}\ar[d]^{\ab^1_H}&\cohomorf{1}{K,G}\ar@{-->}[d]^{\ab^1_G}\\
	\cohomorab{1}{K,F}\ar[r]&\cohomorab{1}{K,H}\ar[r]_{p_*}&\cohomorab{1}{K,G}}.
	\end{equation*}Notice $H\in\cR^o$, so the map $\ab^1_H$ is defined. Since $f$ comes from an element $h\in\cohomorf{1}{K,H}$, this forces us to define $\ab^1_G(f)=p_*(\ab^1_H(h))$. Also notice that $\cohomorf1{K,F}=0$ by Shapiro's lemma and Hilbert 90, such lifting $h$ of $f$ is unique. 
	
	Furthermore, the construction of $\ab^1(f)$ does not rely on the choice of $z$-extension $H\rightarrow G$. In fact, let $p_1:H_1\rightarrow G$ and $p_2:H_2\rightarrow G$ be two $f$-lifting $z$-extensions. Set $H=H_1\times_GH_2$ (the fibre product). Then $p:H\rightarrow G$ is surjective and $\kernel p=\kernel{p_1}\times\kernel{p_2}$. We see that $p$ is a $z$-extension. Since the set of cocycles $Z^1(K,H)$ is the fibre product of $Z^1(K,H_1)$ and $Z^1(K,H_2)$, we conclude that $p$ is also an $f$-lifting $z$-extension. Let $h_1,h_2$ and $h$ be the lifting of $f$ in $\cohomorf{1}{K,H_1},\cohomorf{1}{K,H_2}$ and $\cohomorf{1}{K,H}$ respectively and let $q_1:H\rightarrow H_1$ and $q_2:H\rightarrow H_2$ be the canonical projective, then $h_1=q_{1,*}(h)$ and $h_2=q_{2,*}(h)$ by the uniqueness of a lifting. Therefore, \begin{equation*}
	p_{1,*}(\ab_{H_1}^1(h_1))=p_{1,*}(q_{1,*}(\ab_{H}^1(h)))=p_{2,*}(q_{2,*}(\ab_{H}^1(h)))=p_{2,*}(\ab_{H_2}^1(h_2)).
	\end{equation*}Hence, $\ab^1_G(f)$ is well-defined.
	
	Now, the proof of Theorem \ref{abthm1} is clear.
	
	
	
	These  maps $\ab^1$ define morphism of cohomology exact sequences in the following sense.
	
		\begin{theorem}[\cite{Bo89}, Proposition 3.10 or \cite{Bo98}, Proposition 3.12]\label{abthm2}
		Let \begin{equation*}
			\xymatrix{1\ar[r]& G_1\ar[r]& G_2\ar[r]& G_3\ar[r]& 1}
			\end{equation*}be an exact sequence of connected reductive $K$-groups.
			
		(1) Then the diagram \begin{equation*}
			\xymatrix{\cohomorf{1}{K,G_1}\ar[r]\ar[d]^{\ab^1}&\cohomorf{1}{K,G_2}\ar[r]\ar[d]^{\ab^1}&\cohomorf{1}{K,G_3}\ar[d]^{\ab^1}\\
			\cohomorab{1}{K,G_1}\ar[r]&\cohomorab{1}{K,G_2}\ar[r]&\cohomorab{1}{K,G_3}}
			\end{equation*}is commutative and rows are exact.
			
		(2) If moreover $G_1$ is a torus, then the diagram 
		\begin{equation*}
				\xymatrix{\cohomorf{1}{K,G_2}\ar[r]\ar[d]^{\ab^1}&\cohomorf{1}{K,G_3}\ar[d]^{\ab^1}\ar[r]&\cohomorf{2}{K,G_1}\ar[d]_{\sim}^{\ab^2}\\
				\cohomorab{1}{K,G_2}\ar[r]&\cohomorab{1}{K,G_3}\ar[r]_{\delta_1}&\cohomorab{2}{K,G_1}}
				\end{equation*}is commutative and rows are exact.
		
		\end{theorem}
		
		
\subsection{Base-change of the abelianization maps}\,
	
	Let $L$ be a field extension of $K$ and $G$ a connected reductive group over $K$. Let $Z^\bullet$ be the complex associated to the centers of $G^\bullet$ in Definition \ref{cdef}, then group $Z_L$ (resp. $Z_L^{(sc)}$) is also the center of $G_L$ (resp. $G_L^{sc}$). Hence, one induces a restriction homomorphism \begin{equation*}
	\cohomorh{i}{K,Z^\bullet}=\cohomorab{i}{K,G}\rightarrow\cohomorab{i}{L,G}=\cohomorh{i}{L,Z_L^\bullet}
	\end{equation*}of cohomological $\delta$-functors. It is compatible with the canonical Galois cohomologies in the following sense.
		
	\begin{lemma}\label{com1}
	Let $L$ be a field extension of $K$. Let $G$ be a connected reductive $K$-group, then the diagram
	\begin{equation*}
	\xymatrix{
	\cohomorf{1}{K,G}\ar[r]\ar[d]^{\ab^1}&\cohomorf{1}{L,G}\ar[d]^{\ab^1}\\
	\cohomorab{1}{K,G}\ar[r]&\cohomorab{1}{L,G}}
	\end{equation*} is commutative.
	\end{lemma}
	
	\begin{proof}
	Take an element $f\in\cohomorf{1}{K,G}$ and an $f$-lifting $z$-extension $H\xrightarrow pG$ by Lemma \ref{exist2}, one reduces to the case that $G\in\cR^o$. But this is clear since the diagram is commutative for tori.
	\end{proof}
	
	Now, let $L$ be a finite extension of $K$, according to \cite[Sect 2.3]{Del79} or \cite[Char XVIII, Sect 6.3]{SGA43}, one has the following commutative diagram \begin{equation*}
			\xymatrix{
			\mathrm{Res}_{L/K}Z_L\ar[r]^{\quad \mathrm{Tr}_{L/K}}&Z\\
			\mathrm{Res}_{L/K}Z^{(sc)}_L\ar[u]\ar[r]_{\quad \mathrm{Tr}_{L/K}}&Z^{(sc)}\ar[u]}.
			\end{equation*}This is a morphism of complexes of $\Gamma_K$-modules, then it induces a corestriction homomorphism
	\begin{equation*}\cor{L/K}:\cohomorab{1}{L,G}\rightarrow\cohomorab{1}{K,G}
	\end{equation*}of Galois hyper-cohomology with  the help of Shapiro's lemma. It satisfies a similar restriction-corestriction property as the canonical case in the following sense.
	
	\begin{lemma}\label{res-cor}
	Let $L$ be a finite extension of $K$, $G$ a connected reductive $K$-group, then the composition of following homomorphisms \begin{equation*}
	\xymatrix{\cohomorab{1}{K,G}\ar[r]&\cohomorab{1}{L,G}\ar[r]^{\cor{L/K}}&\cohomorab{1}{K,G}}
	\end{equation*}is the endomorphism of $\cohomorab{1}{K,G}$ by multiplicating the integer number $[L:K]$. 
	\end{lemma}
	
	\begin{proof}
	It follows from the fact that the following diagram is commutative \begin{equation*}
	\xymatrix{Z\ar[r]\ar@/_1pc/[rr]_{[L:K]}&\wres{L/K}Z_L\ar[r]^{\quad\mathrm{Tr}_{L/K}}&Z}
	\end{equation*}for any commutative $k$-group $Z$.
	\end{proof}
	
	These corestriction homomorphisms form a morphism of functors in the following sense.
	
	\begin{lemma}
	Let $L$ be a finite extension of $K$, let $p:G_1\rightarrow G_2$ be a morphism of connected reductive $K$-groups, then the following diagram is commutative
	\begin{equation*}
	\xymatrix{\cohomorab{1}{L,G_1}\ar[r]^{p_*}\ar[d]^{\cor{L/K}}&\cohomorab{1}{L,G_2}\ar[d]^{\cor{L/K}}\\
	\cohomorab{1}{K,G_1}\ar[r]_{p_*}&\cohomorab{1}{K,G_2}}.
	\end{equation*}
	\end{lemma}
	
	\begin{proof}
	Let $T_1$ be a maximal torus of $G_1$ and $T_2$ be a maximal torus of $G_2$ containing the image of $T_1$ in $G_2$. Then one has a commutative diagram\begin{equation*}
	\xymatrix{\wres{L/K}(T^\bullet_1)\ar[d]\ar[r]^{\quad\quad\mathrm{Tr}_{L/K}}&T^\bullet_1\ar[d]\\
	\wres{L/K}(T^\bullet_2)\ar[r]^{\quad\quad\mathrm{Tr}_{L/K}}&T^\bullet_2}.
	\end{equation*}Since the definition of $p_*$ is independent on the choice of maximal tori $T_1$ or $T_2$, one concludes by applying the hyper-cohomology on the above diagram with the help of Shapiro's lemma.
	\end{proof}

	We now care about the commutativity between corestriction maps and connecting maps of abelianized cohomologies. To do so, we give an abelianized version of Lemma \ref{exist2} at first.
	
	\begin{lemma}\label{exist3} Let $G$ be a connected reductive $K$-group, $L$ a finite extension of $K$, $f\in\cohomorab{1}{L,G}$ be an element, then there exists a $z$-extension $H\rightarrow G$ over $K$ such that $f$ comes from $\cohomorab{1}{L,H}$.  
	
	\end{lemma} 
	
	\begin{proof}
	We only need to notice that $\cohomorab{1}{\overline K,G}=0$ (because the hyper-Galois cohomology over $\overline K$ is trivial). Then there is also a finite extension $L'$ of $L$ such that $L'$ is a Galois extension of $K$, $f$ maps to $0\in\cohomorab{1}{L',G}$ and $L'$ splits $G$. The rest of the proof is the same as the one of Lemma \ref{exist2} except replacing all $\cohomorf{1}{*,*}$ by $\cohomorab{1}{*,*}$.
	\end{proof}
	
	\begin{lemma}\label{cor-com}Let $L$ be a finite extension of $K$, \begin{equation*}
			\xymatrix{1\ar[r]& G_1\ar[r]& G_2\ar[r]& G_3\ar[r]& 1}
			\end{equation*}an exact sequence of connected reductive linear  $K$-groups where $G_1$ is a torus, then the diagram\begin{equation*}
			\xymatrix{\cohomorab{1}{L,G_3}\ar[d]^{\cor{L/K}}\ar[r]^{\delta_1}&\cohomorab{2}{L,G_1}\ar[d]^{\cor{L/K}}\\
			\cohomorab{1}{K,G_3}\ar[r]_{\delta_1}&\cohomorab{2}{K,G_1}}
			\end{equation*}is commutative.
	\end{lemma}

	\begin{proof}
	Take an element $f\in\cohomorab{1}{L,G_3}$, then by Lemma \ref{exist3}, there is a $z$-extension $H_3\rightarrow G_3$ such that $f$ comes from $\cohomorab{1}{L,H_3}$. Let $G_2'$ be the cartesian product of the following diagram \begin{equation*}
	\xymatrix{G_2'\ar[r]\ar[d]&H_3\ar[d]\\
	G_2\ar[r]&G_3}
	\end{equation*}and take a $z$-extension $H_2\rightarrow G_2'$, then we have a commutative diagram \begin{equation*}
	\xymatrix{1\ar[r]& H_1\ar[r]\ar[d]& H_2\ar[d]\ar[r]& H_3\ar[r]\ar[d]& 1\\
	1\ar[r]& G_1\ar[r]& G_2\ar[r]& G_3\ar[r]& 1}
	\end{equation*}with exact rows such that $H_2,H_3\in\cR^o$ and $H_1$ is a torus (it is a central extension of two tori). Then the question reduces to the case where $G_2,G_3$ lie in $\cR^o$. Now, consider the following diagram\begin{equation}\label{??1}
		\xymatrix{1\ar[r]& G_1\ar[r]\ar[d]& G_2\ar@{->>}[d]\ar[r]& G_3\ar[r]\ar@{->>}[d]& 1\\
		1\ar[r]& G_1'\ar[r]& G_2^{tor}\ar[r]& G_3^{tor}\ar[r]& 1}
		\end{equation}where rows are exact and $G_1'$ is a torus (it is an extension of two connected groups and is also a group of multiplicative type). By five lemma, $\cohomorab{1}{K,G_1}\rightarrow\cohomorab{1}{K,G_1'}$ and $\cohomorab{1}{L,G_1}\rightarrow\cohomorab{1}{L,G_1'}$ are isomorphisms. Since the lemma holds for the second row of \reff{??1}, it also holds for the first row.
	\end{proof}
	
	\begin{remark}Lemma \ref{cor-com} holds for a general $L$, but the proof requires other unnecessary lemmas. We only require the case that $L$ is a number field or a local field. So it is fine.
	\end{remark}

\subsection{Abelian Galois cohomology of connected linear groups over a number field}\,

	We now consider a connected linear (non-necessarily reductive) group $G$ over a number field or a local field $K$. Let $R_u(G)$ be the unipotent radical of $G$ and set $G^{red}=G/R_u(G)$ (it is reductive).

	\begin{lemma}[\cite{Har02}, the proof of Theorem 2]\label{ab:mod}
	Let $K$ be  a number field or a local field, $G$ be a connected linear $K$-group, then the restriction map\begin{equation*}
	\cohomorf{1}{K,G}\simeq\cohomorf{1}{K,G^{red}}
	\end{equation*}is an isomorphism.
	\end{lemma}
	
	Due to this lemma, we set $\cohomorab{i}{K,G}=\cohomorab{i}{K,G^{red}}$ and let $\ab^1_G$ denote the composition of morphisms\begin{equation*}
	\cohomorf{1}{K,G}\rightarrow\cohomorf{1}{K,G^{red}}\xrightarrow{\ab^1_{G^{red}}}\cohomorab{1}{K,G}.
	\end{equation*}
	
	\begin{lemma}\label{ab:surj}
	Let $K$ be a number field or a local field, $G$ be a connected linear $K$-group, then the abelianization map\begin{equation*}
	\ab^1:\cohomorf{1}{K,G}\twoheadrightarrow\cohomorab{1}{K,G}
	\end{equation*}is surjective. Moreover, if $K$ is a non-archimedian local field, then $\ab^1$ is bijective.
	\end{lemma}
	
	\begin{proof}
	By Lemma \ref{ab:mod}, it reduces to the case of reductive connected group, then apply \cite[Theorem 5.4, Corollary 5.4.1 and Theorem 5.7]{Bo98}.
	\end{proof}
	
	Now, let $k$ be a number field and $v\in\Omega_k$ be a place. Let $\loc_v,\loc_v^\ab,\loc$ and $\loc^\ab,$ denote the canonical restriction map \begin{equation*}
	\loc_v:\cohomorf{1}{k,G}\rightarrow\cohomorf1{k_v,G}\quad\quad\loc_v^\ab:\cohomorab{1}{k,G}\rightarrow\cohomorab1{k_v,G}
	\end{equation*}\begin{equation*}
		\loc:\cohomorf{1}{k,G}\rightarrow\prod\limits_{v\in\Omega_k}\cohomorf1{k_v,G}\quad\quad\loc^\ab:\cohomorab{1}{k,G}\rightarrow\prod\limits_{v\in\Omega_k}\cohomorab1{k_v,G}.
		\end{equation*}Let $\ab^1_v$ denote the abelianization map $\cohomorf{1}{k_v,G}\rightarrow\cohomorab{1}{k_v,G}$. These maps fit into commutative diagrams 	
	\begin{equation*}
	\xymatrix{
	\cohomorf{1}{k,G}\ar[r]^{\loc_v}\ar[d]_{\ab^1}&\cohomorf1{k_v,G}\ar[d]^{\ab^1_v}
		\\
			\cohomorab{1}{k,G}\ar[r]_{\loc_v^\ab}&\cohomorab1{k_v,G}}
	\end{equation*}for each $v\in\Omega_k$ by Lemma \ref{com1}.

	\begin{theorem}\label{iff}
	Let $G$ be a conneccted linear group over a number field $k$, then the commutative diagram \begin{equation*}
	\xymatrix{
	\cohomorf{1}{k,G} \ar@{->>}[r]^{\ab^1}\ar[d]_{\loc}&\cohomorab{1}{k,G}\ar[d]^{\loc^\ab}\\
	\prod\limits_{v\in\Omega_k}\cohomorf{1}{k_v,G}\ar@{->>}[r]_{\ab^1}&\prod\limits_{v\in\Omega_k}\cohomorab{1}{k_v,G}
	}
	\end{equation*}is cartesian.
	\end{theorem}
	
	\begin{proof}
	We may assume that $G$ is reductive by Lemma \ref{ab:mod}, then it is the corollary of \cite[Theorem 5.12]{Bo98} and Lemma \ref{ab:surj}.
	\end{proof}

%
	Next, we introduce an exact sequence for a connected linear group $G$ over a number field $k$.
	
	Since the characteristic of $k$ is 0, $G$ is isomorphic to $G^{red}\times_k R_u(G)$ as a $k$-variety and $R_u(G)$ is isomorphic to an affine space. Therefore, the canonical restriction $\pic{ G^{red}}\rightarrow\pic{G}$ is bijective.
	
	Any element $E\in\pic G$ can be given the structure of a central extension of algebrac groups \begin{equation}\label{seq11}
	1\rightarrow\bG_m\rightarrow E\rightarrow G\rightarrow 1
	\end{equation}by \cite[Corollary 5.7]{CT08}. It comes from a unique element $E_0\in\pic{G^{red}}$ and it induces a commutative diagram \begin{equation}\label{seq2}
	\xymatrix{1\ar[r]&\bG_m\ar[r]\ar[d]&E\ar[r]\ar[d]&G\ar[r]\ar[d]&1\\
	1\ar[r]&\bG_m\ar[r]&E_0\ar[r]&G^{red}\ar[r]&1}
		\end{equation}where rows are exact and the right square is a cartesian diagram. One obtains a coboundary map $\partial_E^X:\cohomorf{1}{X,G}\rightarrow\cohomorf{2}{X,\bG_m}=\brgp X$ associated to $E$. 
	
	Let $f:Y\rightarrow X$ be a torsor under $G$,  then the homomorphism $\Delta_{Y/X}:\pic{G}\rightarrow\brgp X,E\mapsto\partial_E^X(f)$ induces an exact sequence \begin{equation}\label{fund2}
	\xymatrix{\pic{G}\ar[r]^{\Delta_{Y/X}}&\brgp X\ar[r]^{f^*}&\brgp Y}
	\end{equation}by \cite[Theorem 2.8]{BoD13}. 
	
	For the number field $k$ or any local field $L$ containing $k$, the commutative diagram \reff{seq2} also induces commutative diagrams \begin{equation}\label{comm1}
	\xymatrix{\cohomorf1{k,G}\ar[r]^{\partial_E^k\quad\quad}\ar[d]_{\ab^1}&\cohomorf2{k,\bG_m}=\brgp{k}\ar@{=}[d]\\
	\cohomorab1{k,G}\ar[r]_{\partial_E^{k,\ab}\quad\quad}&\cohomorab2{k,\bG_m}=\brgp{k}}\quad\quad\xymatrix{\cohomorf1{L,G}\ar[r]^{\partial_E^L\quad\quad}\ar[d]_{\ab^1}&\cohomorf2{L,\bG_m}=\brgp{L}\ar@{=}[d]\\
		\cohomorab1{L,G}\ar[r]_{\partial_E^{L,\ab}\quad\quad}&\cohomorab2{L,\bG_m}=\brgp{L}}
	\end{equation}by Theorem \ref{abthm2} and Lemma \ref{ab:mod}. Here, the homomorphism $\partial_E^{k,\ab}$ or $\partial_E^{L,\ab}$ is no other than the connecting homomorphism $\partial_{E_0}^{k,\ab}$ or $\partial_{E_0}^{L,\ab}$ of abelian Galois cohomologies \begin{equation*}
	\xymatrix{\cohomorab{1}{k,G}=\cohomorab{1}{k,G^{red}}\ar[r]^{\quad\quad\quad\partial_{E_0}^{k,\ab}}&\cohomorab{2}{k,\bG_m}\\
	\cohomorab{1}{L,G}=\cohomorab{1}{L,G^{red}}\ar[r]^{\quad\quad\quad\partial_{E_0}^{L,\ab}}&\cohomorab{2}{L,\bG_m}}
	\end{equation*}induced from the second row of the diagram \reff{seq2}. 
	

	\begin{lemma}\label{inv:map} Let $L$ be a local field, then the pairing\begin{equation}\label{pr1}
	\cohomorab{1}{L,G}\times\pic G\rightarrow\bQ/\bZ,\quad(a,E)\mapsto\inv_L(\partial_E^{L,\ab}(a))
	\end{equation}is a bilinear map. 
	\end{lemma}
	
	\begin{proof}
	Since $\inv_L\circ\partial_E^{L,\ab}$ is a homomorphism  for any given $E\in\pic G$, the pairing is linear at left. On the other hand, the following diagram is commutative by the diagram \reff{comm1} \begin{equation}\label{temp1}
	\xymatrix{\cohomorf{1}{L,G}\times\pic G\ar[r]\ar[d]&\bQ/\bZ\ar@{=}[d]&\quad(a,E)\mapsto\inv_L(\partial_E^{L}(a))\\
	\cohomorab{1}{L,G}\times\pic G\ar[r]&\bQ/\bZ&\quad(a,E)\mapsto\inv_L(\partial_E^{L,\ab}(a)).}
	\end{equation}By \cite[Proposition 2.9]{CT09}, the first row induces a map $\cohomorf{1}{L,G}\rightarrow\hom{}(\pic G,\bQ/\bZ)=\pic G^D$. For an element $a\in\cohomorab{1}{L,G}$, we may write $a=\ab^1(a_0)$ for some $a_0\in\cohomorf{1}{L,G}$ by Lemma \ref{ab:surj}, then the map $\pic G\rightarrow\bQ/\bZ,\,E\mapsto\inv_L(\partial_E^{L,\ab}(a))$ assciated to $f$ is no other than the map $E\mapsto\inv_L(\partial_E^{L}(a_0))$ by the commutative diagram \reff{temp1}. Hence, the pairing is linear at right.
		\end{proof}
	

	Given a place $v\in\Omega_k$ and a finite field extension $L$ of $k_v$, let $\mu_L$  denote the map \begin{equation*}
		\mu_L:\cohomorf{1}{L,G}\rightarrow\pic G^{D}
		\end{equation*}and $\mu_L^\ab$ denote the homomorphism\begin{equation*}
		\mu_L^\ab:\cohomorab{1}{L,G}\rightarrow\pic G^{D}.
		\end{equation*} Especially, denote $\mu_v=\mu_{k_v}$ and $\mu_v^\ab=\mu_{k_v}^\ab$. Since for each $E\in\pic G$, there is a commutative diagram\begin{equation*}
	\xymatrix{
	\cohomorab{1}{L,G}\ar[r]^{\partial_E^{L,\ab}}\ar[d]_{\cor{L/k_v}}&\brgp L\ar[r]^{\inv_L}\ar[d]_{\cor{L/k_v}}&\bQ/\bZ\ar@{=}[d]\\
	\cohomorab{1}{k_v,G}\ar[r]_{\partial_E^{k_v,\ab}}&\brgp{k_v}\ar[r]_{\inv_{k_v}}&\bQ/\bZ}
	\end{equation*}where the commutativity of the left square is guaranteed by Lemma \ref{cor-com}, one obtains a commutative diagram
		\begin{equation}\label{cor-com2}
		\xymatrix{\cohomorab{1}{L,G}\ar[r]^{\mu_L^\ab}\ar[d]_{\cor{L/k_v}}&\pic{G}^D\ar@{=}[d]\\
		\cohomorab{1}{k_v,G}\ar[r]_{\mu_v^\ab}&\pic{G}^D}.
		\end{equation}
		
	We end this section by the following lemma.
	
	\begin{lemma}\label{fund:ext}Let $G$ be a connected linear group over a number field $k$, then the image of $\loc^\ab$ lies in $\bigoplus\limits_{v\in\Omega_k}\cohomorab{1}{k_v,G}$ and the following sequence is exact
	\begin{equation}\label{fund:exact}
	\cohomorab{1}{k,G}\xrightarrow{\loc^\ab}\bigoplus\limits_{v\in\Omega_k}\cohomorab{1}{k_v,G}\xrightarrow{\sum\mu_v^\ab}\pic G^D.
	\end{equation}
	\end{lemma}
	
	\begin{proof}
	
Consider the following commutative diagram \begin{equation*}
	\xymatrix{
	\cohomorf{1}{k,G}\ar[r]^{\loc\quad}\ar@{->>}[d]_{\ab^1}&\prod\limits_{v\in\Omega_k}\cohomorf{1}{k_v,G}\ar[d]\\
	\cohomorab{1}{k,G}\ar[r]_{\loc^\ab\quad}&\prod\limits_{v\in\Omega_k}\cohomorab{1}{k_v,G}}.
	\end{equation*}Since the image of $\loc$ lies in $\bigoplus\limits_{v\in\Omega_k}\cohomorf{1}{k_v,G}$ by Lang's theorem and Hensel's lemma, the image of $\loc^\ab$ lies in $\bigoplus\limits_{v\in\Omega_k}\cohomorab{1}{k_v,G}$ by Lemma \ref{ab:surj}.
	By \cite[Theorem 3.1]{CT09}, the first row in the following commutative diagram is exact\begin{equation*}
	\xymatrix{\cohomorf{1}{k,G}\ar@{->>}[d]_{\ab^1}\ar[r]^{\loc\quad}&\bigoplus\limits_{v\in\Omega_k}\cohomorf{1}{k_v,G}\ar@{->>}[d]_{\oplus\ab^1}\ar[r]^{\quad\sum\mu_v}&\pic G^D\ar@{=}[d]\\
	\cohomorab{1}{k,G}\ar[r]_{\loc^\ab\quad}&\bigoplus\limits_{v\in\Omega_k}\cohomorab{1}{k_v,G}\ar[r]_{\quad\sum\mu_v^\ab}&\pic G^D}.
	\end{equation*}The the second row is exact by Lemma \ref{ab:surj} and diagram chasing.
	\end{proof}
	
\section{descent obstruction for a connected linear group}

	Let $X$ be a smooth and geometrically integral variety over a number field $k$. Let $G$ be a connected linear $k$-group. Firstly, we define the abelianized (local) evaluations
	\begin{equation*}
		(\cdot,\cdot)^v_{\ab}:\cohomorf{1}{X,G}\times\zcycle{}{}(X_v)\rightarrow \cohomorab{1}{k_v,G}
		\end{equation*}as follows.

	Let $f\in\cohomorf{1}{X,G}$ be a torsor under $G$ over $X$. 
	
	If $P_v$ is a closed point on $X_v$, we send $f$ to $\cor{k(P_v)/k_v}(\text{ab}^1(f(P_v)))$. Alternatively speaking, it is the image of $f$ along the  composition of the following maps \begin{equation*}
	\cohomorf{1}{X,G}\xrightarrow{(P_v)}\cohomorf{1}{k(P_v),G}\xrightarrow{\ab^1}\cohomorab{1}{k(P_v),G}\xrightarrow{\cor{k(P_v)/k_v}}\cohomorab{1}{k_v,G}.
	\end{equation*}
	
	If $x_v=\sum\limits_in_i.P_{v,i}$ is a general 0-cycle on $X_v$, then define \begin{equation*}
		(f,x_v)_\ab^v=\sum\limits_{i}n_{i}.(f,P_{v,i})_\ab^v.
		\end{equation*}
	
	With the abelianized evaluation $(\cdot,\cdot)_\ab^v$ at each place $v\in\Omega_k$, we define the global piaring
\begin{equation}\label{ab:pair}
	\ABcouple{\cdot,\cdot}:\cohomorf{1}{X,G}\times\zcyclef{}{\vA}(X)\rightarrow\pic G^D
	\end{equation}\begin{equation}\label{def:sum}
	\ABcouple{f,(x_v)_{v\in\Omega_k}}=\sum\limits_{v\in\Omega_k}\mu_v^\ab(f,x_v)_\ab^v
	\end{equation}for any $f\in\cohomorf{1}{X,G}$ and any adelic 0-cycle $(x_v)_{v\in\Omega_k}\in\zcyclef{}{\vA}(X)$. The following  lemma shows the sum is indeed a finite sum, hence, well-defined.
	
	\begin{lemma}\label{loc:tv3}
	Let $X$ be a smooth and geometrically integral variety over a number field $k$, $G$ a connected linear $k$-group. Let $f\in\cohomorf{1}{X,G}$ be a torsor. Let $(x_v)_{v\in\Omega_k}$ be an adelic 0-cycle on $X$, then there is a finite subset $S\subseteq\Omega_k$ (dependent on $f$ and $(x_v)_{v\in\Omega_k}$) such that for any $v\not\in S$, one obtains that  $(f,x_v)_\ab^v=0$. 
	
	\end{lemma}
	
	\begin{proof}Take a finite subset $S\subseteq\Omega_k$ as in Lemma \ref{loc:tv2}. For each $v\not\in S$ and each closed point $P_v$ lying in the support of $x_v$, one obtains that $\ab^1(f(P_{v}))=0\in\cohomorab{1}{k(P_v),G}$. Hence, one obtains that $(f,x_v)_\ab^v=0$ for $v\not\in S$.\end{proof}
	
	Let $f\in\cohomorf{1}{X,G}$ be a torsor. We set \begin{equation*}
	\zcyclef{}{\vA}(X)^f_\ab=\{(x_v)_{v\in\Omega_k}\in\zcyclef{}{\vA}(X):\ABcouple{f,(x_v)_{v\in\Omega_k}}=0\}.
	\end{equation*}By the exact sequence \reff{fund:exact}, this definition is no other than the one defined in Introduction. The following lemma shows the relation between $X(\vA)^f$ and $\zcyclef{}{\vA}(X)^f_\ab$.
	
	\begin{lemma}\label{rat-cyc}
	Let $X$ be a smooth and geometrically integral variety over a number field $k$, $G$ a connected linear $k$-group and $f\in\cohomorf{1}{X,G}$ a torsor. Let $(P_v)_{v\in\Omega_k}\in X(\vA)$ be an adelic point, then $(P_v)_{v\in\Omega_k}\in X(\vA)^f$ if and only if $(P_v)_{v\in\Omega_k}\in\zcyclef{}{\vA}(X)^f_\ab$, regarding $(P_v)_{v\in\Omega_k}$ as an adelic 0-cycle on $X$ of degree 1.
	\end{lemma}
	
	\begin{proof}
	First notice that, by Lemma \ref{loc:tv3}, $((f,P_v)_\ab^v)_{v\in\Omega_k}$ lies in $\bigoplus\limits_{v\in\Omega_k}\cohomorab{1}{k_v,G}$. By the exact sequence \reff{fund:exact}, $\ABcouple{f,(P_v)_{v\in\Omega_k}}=0$ if and only if $((f,P_v)_\ab^v)_{v\in\Omega_k}$ lies in the image of $\cohomorab{1}{k,G}\rightarrow\bigoplus\limits_{v\in\Omega_k}\cohomorab{1}{k_v,G}$. Then the lemma follows from Theorem \ref{iff}.
	\end{proof}
	
	Now, we prove a theorem which is slightly more general than Theorem \ref{main3}.

	\begin{theorem}\label{main2}
		Let $k$ be a number field, $G$ a connected linear $k$-group, and $X$ a smooth and geometrically integral $k$-variety. Let $f:Y\rightarrow X$ be a  torsor under $G$ and let $\br'_f$ be the kernel of restriction map $f^*:\brgp X\rightarrow\brgp Y$, then we have \begin{equation*}
		\zcyclef{}{\vA}(X)^{\br'_f}=\zcyclef{}{\vA}(X)^f_\ab.
		\end{equation*}(We only use the prime of the symbol $\br'_f$ to distinguish the Brauer subgroup $\br_\lambda$ in Section 2.)
		\end{theorem}

	\begin{proof}
	Let $(x_v)_{v\in\Omega_k}$ be an adelic 0-cycle, then $(x_v)_{v\in\Omega_k}\in\zcyclef{}{\vA}(X)^f$ if and only if $((f,x_v)_\ab^v)_{v\in\Omega_k}$ lies in the image of $\loc:\cohomorab{1}{k,G}\rightarrow\bigoplus\limits_{v\in\Omega_k}\cohomorab{1}{k_v,G}$. For any given $E\in\pic G$, the following diagram is commutative
	
	\begin{equation*}
	\xymatrix{
	\cohomorf{1}{X,G}\ar[d]_{\partial_E}\ar[r]^{(x_v)\quad}&\bigoplus\limits_{v\in\Omega_k}\cohomorab{1}{k_v,G}\ar[d]_{\oplus\partial_E^{k_v,\ab}}\ar[r]^{\quad\quad\sum\mu_v^\ab}&\pic G^D\ar[d]^{(E)}\\
	\brgp X\ar[r]_{(x_v)\quad}&\bigoplus\limits_{v\in\Omega_k}\brgp{k_v}\ar[r]_{\quad\quad\sum\inv_v}&\bQ/\bZ}.
	\end{equation*} Therefore, one obtains that\begin{equation*}
	\BMcouple{\Delta_{Y/X}(E),(x_v)_{v\in\Omega_k}}=\ABcouple{f,(x_v)_{v\in\Omega_k}}(E).
	\end{equation*}The exact sequence \reff{fund2} shows that the image of homomorphism $\Delta_{Y/X}:\pic G\rightarrow\br(X)$ is no other than $\br_f'$. Then $\ABcouple{f,(x_v)_{v\in\Omega_k}}=0$ if and only if $(x_v)_{v\in\Omega_k}$ is orthogonal to $\br_f'$.
	\end{proof}



\begin{remark}
Notice that given a torsor $\sigma\in\cohomorf{1}{k,G}$, the desncet subset $X(\vA)^{f^\sigma}$ is the same as $X(\vA)^{f}$ due to the $f$-descent decomposition  \reff{g-inter}. For 0-cycles, the descent subgroup $\zcyclef{}{\vA}(X)^{f^\sigma}_\ab$ is also independent on the twist by an element $\sigma\in\cohomorf{1}{k,G}$. This observation is based on Theorem \ref{main2} and the next lemma.
\end{remark}

\begin{lemma}\label{brf-dec}
Let $X$ be a smooth and geometrically integral variety over a number field $k$, $G$ a connected linear $k$-group and $f:Y\rightarrow X$ be a left torsor under $G$. Let $\sigma\in\cohomorf{1}{k,G}$ be an element and $f^\sigma:Y^\sigma\rightarrow X$ the twist of $f$ by $\sigma$, then \begin{equation*}
\br_f'+\br_0(X)=\br_{f^\sigma}'+\br_0(X).
\end{equation*}
\end{lemma}

\begin{proof}Let $P$ be the left torsor associated to $\sigma$, then $Y^\sigma$ is the quotient of $P\times_k Y$ by diagonal action of $G$ given by $g.(p,y)=(g.p,g.y)$. The canonical projection $\xi_P:P\times_kY\rightarrow Y^\sigma$ induces a diagram\begin{equation}\label{brf-1}
\xymatrix{P\times_kY\ar[r]^{\xi_P}\ar[d]_{\text{pr}_Y}&Y^\sigma\ar[d]^{f^\sigma}\\
Y\ar[r]_f&X}
\end{equation}

Firstly assume that $P$ is trivial. Any rational point $\theta\in P$ determines an isomorphism\begin{equation*}
\zeta_\theta:Y\xrightarrow{y\mapsto(\theta,y)} P\times Y\xrightarrow{\xi_P}Y^\sigma.
\end{equation*}This isomorphism gives rise to a commutative diagram \begin{equation*}
\xymatrix{\br(X)\ar[rd]_{f^\ast}\ar[r]^{f^{\sigma,\ast}}&\br(Y^\sigma)\ar[d]_\sim^{\zeta_\theta^\ast}\\
&\br(Y)}
\end{equation*}

Return to a general $P$. Since $P$ is trivial over $\overline k$ and the diagram \reff{brf-1} is compatible with field extension, this tells us that $(f^*)^{-1}(\br_1(Y))=(f^{\sigma,*})^{-1}(\br_1(Y^\sigma))$. Let $B$ denote this subgroup, then $B$ contains $\br_0(X),\br_f'$ and $\br_{f^\sigma}'$.

Since the canonical homomorphism\begin{equation*}
\br_a(P)\dsum\br_a(Y)\xrightarrow{\sim}\br_a(P\times_kY)
\end{equation*}is isomorphism by \cite[Lemma 5.1]{Bo12}, one has a canonical retraction \begin{equation*}
r:\br_a(P\times_kY)\xleftarrow{\sim}\br_a(P)\dsum\br_a(Y)\rightarrow\br_a(Y)
\end{equation*}and a commutative diagram \begin{equation*}
\xymatrix{
B\ar[r]^{'f^{\ast}}\ar[d]_{'f^{\sigma,\ast}}&\br_a(Y)\ar[d]^{\text{pr}_Y^\ast}\ar[rd]^{\text{id}}\\
\br_a(Y^\sigma)\ar[r]^{\chi_P^\ast}\ar@/_1pc/[rr]_{\psi_P}&\br_a(P\times_kY)\ar[r]^r&\br_a(Y)}.
\end{equation*}Here $'f^\ast$ and $'f^{\sigma,\ast}$ are the homomorphism \begin{equation*}
'f^\ast:B=(f^*)^{-1}(\br_1(Y))\xrightarrow{f^\ast}\br_1(Y)\twoheadrightarrow\br_a(Y)
\end{equation*}\begin{equation*}
'f^{\sigma,\ast}:B=(f^{\sigma,*})^{-1}(\br_1(Y^\sigma))\xrightarrow{f^{\sigma,\ast}}\br_1(Y^\sigma)\twoheadrightarrow\br_a(Y^\sigma)
\end{equation*}induced from $f^\ast$ and $f^{\sigma,\ast}$. The homomoprhism $\psi_P=r\circ\chi_P^\ast$ is an isomorphism by \cite[Lemma 2.7]{Cao19}, then \begin{equation*}
\br_f'+\br_0(X)=\kernel{'f^*}=\kernel{'f^{\sigma,*}}=\br_{f^\sigma}'+\br_0(X).
\end{equation*}
\end{proof}

To prove Theorem \ref{main3}, it is sufficient to prove the following lemma.

\begin{lemma}\label{lemma???}
Let  $X$ be a smooth geometrically integral variety over a number field $k$, $n$ an integer number $\geq2$, $\alpha\in\br(X)[n]$ an element  of torsion $n$, then there is a  torsor $f:Y\rightarrow X$ under $\mathrm{PGL}_n$ such that $\alpha\in\br_f'$.
\end{lemma}

\begin{proof}
The canonical exact sequence \begin{equation*}
\xymatrix{1\ar[r]&\bG_m\ar[r]&\mathrm{GL}_n\ar[r]&\mathrm{PGL}_n\ar[r]&1}
\end{equation*}determines a class $E\in\pic{\mathrm{PGL}_n}$ and a connecting map
\begin{equation*}
\xymatrix{\cohomorf{1}{X,\mathrm{PGL}_n}\ar[r]^{\quad\partial_E^X}&\br(X).}
\end{equation*}The element $\alpha$ comes from some $f\in\cohomorf{1}{X,\mathrm{PGL}_n}$ by Gabber's theorem. Let $f:Y\rightarrow X$ be the torsor associated to $f$. Now $\alpha=\partial_E^X(f)=\Delta_{Y/X}(E)$ lies in $\br_f'$ by the exact sequence \reff{fund2}.
\end{proof}

\begin{theorem}\label{main3}
Let $X$ be a smooth and geometrically integral variety over a number field $k$, then\begin{equation*}
\zcyclef{}{\vA}(X)^\br=\zcyclef{}{\vA}(X)^{\mathrm{conn}}_\ab.
\end{equation*}
\end{theorem}
\begin{proof}
Theorem \ref{main2} shows that $\zcyclef{}{\vA}(X)^\br\subseteq\zcyclef{}{\vA}(X)^{\mathrm{conn}}_\ab$. To show the converse, take an adelic 0-cycle $(x_v)_{v\in\Omega_k}\in\zcyclef{}{\vA}(X)^{\mathrm{conn}}$ and an element $\alpha\in\br(X)[n]$ for some positive integer $n\geq2$. Then by Lemma \ref{lemma???}, there is a connected linear group $G$ and an element $f\in\cohomorf1{X,G}$ such that $\alpha\in\br_f'$. The hypothesis $(x_v)_{v\in\Omega_k}\in\zcyclef{}{\vA}(X)_\ab^{\mathrm{conn}}\subseteq\zcyclef{}{\vA}(X)^f_\ab$ shows that $\BMcouple{\alpha,(x_v)_{v\in\Omega_k}}=0$ by Theorem \ref{main2}.
\end{proof}

\section{Topology of adelic 0-cycles}
\subsection{Topology of adelic 0-cycles defined by the weak approximation}\,

Denote $\chow_{0,\vA}(X)=\coprod\limits_{\delta\in\bZ}\Big(\prod\limits_{v\in\Omega_k}\chow_0^\delta(X_v)\Big)$ and $\chow_0^\delta(X_v)$ is the Chow group of degree $\delta$ associated to $X_v$. We first give a topology on $\zcyclef{}{\vA}(X)$ defined by the weak approximation for 0-cycles.

\begin{definition}\label{WA-def}
Let $X$ be a proper variety over a number field $k$. For a subset $S\subseteq\Omega_k$ and a positive integer $n$, denote $A_{n,S}$ the subgroup \{$(x_v)_{v\in\Omega_k}\in\zcyclef{0}\vA(X):$ the image of $x_v$ in $\chow_0(X_v)/n$ vanishes for each $v\in S$\}. 

Given two such pairs $(n,S)$, $(n',S')$, define a partial order $(n,S)\triangleleft(n',S')$ if $n|n'$ and $S\subseteq S'$. Then $A_{n',S'}$ is a subgroup of $A_{n,S}$ if $(n,S)\triangleleft(n',S')$.

Equip $\zcyclef{}\vA(X)$ with the weakest uniform topology \cite[Chapter III, Section 3.1]{BBK07} such that $\zcyclef{}{\vA}(X)$ is a topological group where the collection \{$A_{n,S}:n>0\text{ and }S\subseteq\Omega_k\text{ is finite}$\} forms the neighborhood base of the trivial adelic 0-cycle $0\in\zcyclef{}{\vA}(X)$.
\end{definition}
%

It is easy to check that for any subset $A$ of $\zcyclef{}{\vA}(X)$, the closure of $A$ is determined by the equation\begin{equation*}
\overline A=\bigcap\limits_{n,S}(A+A_{n,S}).
\end{equation*}Here the intersection is taken over all positive integers  $n$ and all finite subsets $S$ of $\Omega_k$ and the symbol $A+A_{n,S}$ is short for \{$(x_v)_{v\in\Omega_k}+(x_v')_{v\in\Omega_k}:(x_v)_{v\in\Omega_k}\in A$ and $(x_v')_{v\in\Omega_k}\in A_{n,S}$\}.

This definition of topology on $\zcyclef{}{\vA}(X)$ just translates the definition of weak approximation for 0-cycles given in \cite{Liang13}. To be precise, given an integer $\delta$ and a subset $A$ of $\zcyclef{\delta}{\vA}(X)$, an adelic 0-cycle $(x_v)_{v\in\Omega_k}\in\zcyclef{\delta}{\vA}(X)$ lies in the closure $\overline A$ if and only if for each positive integer $n$ and each finite subest $S\subseteq\Omega_k$, there is a 0-cycle $(x_v')_{v\in\Omega_k}$ such that $x_v$ and $x_v'$ has the same image in $\chow(X_v)/n$ for each $v\in S$.

In other words, given an integer $\delta$, the variety $X$ satisfies the weak  approximation for 0-cycles of degree $\delta$ (resp. the weak  approximation with the \BM\ obstruction for 0-cycles of degree $\delta$) if and only if \begin{equation*}
\zcyclef{\delta}{\vA}(X)=\overline{\zcycle{\delta}(X)}\text{ (resp. }\zcyclef{\delta}{\vA}(X)^\br=\overline{\zcycle{\delta}(X)}\text{).}
\end{equation*}

The following proposition shows that the \BM\ obstruction or the $f$-descent obstructions considered so far are all closed.

\begin{proposition}\label{closed-ness}Let $X$ be a smooth, proper and geometrically integral variety over a number field $k$.

(1) Let $B$ be a subset of $\br(X)$, then $\zcyclef{}{\vA}(X)^B$ is a closed subgroup of $\zcyclef{}{\vA}(X)$. Furthermore, if there is a positve integer $n$ such that $n.B=$\{$n.\alpha:\alpha\in B$\}$\subseteq\br_0(X)$, then  $\zcyclef{}{\vA}(X)^B$ is open as well.

(2) Let $G$ be a $k$-group of multiplicative type and $f:Y\rightarrow X$ a torsor under $G$, then $\zcyclef{}{\vA}(X)^f$ is a closed subgroup of $\zcyclef{}{\vA}(X)$.

(3) Let $G$ be a connected linear $k$-group and $f:Y\rightarrow X$ a torsor under $G$, then $\zcyclef{}{\vA}(X)^f_\ab$ is a closed subgroup of $\zcyclef{}{\vA}(X)$. Furthermore, if there is a positive integer $n$ such that $n.\br_f'\subseteq\br_0(X)$, then $\zcyclef{}{\vA}(X)^f_\ab$ is open as well.
\end{proposition}

\begin{proof}By Theorem \ref{main1} and Theorem \ref{main3}, (2) and (3) follows from (1). Hence, we only need to claim (1).

Claim (1). We first claim the second assersion of (1).

Fix a finite subset $S$ of $\Omega_k$ such that $\BMcouple{\alpha,(x_v)_{v\in\Omega_k}}=\sum\limits_{v\in S}\inv_v(\alpha,x_v)$ for any adelic 0-cycle $(x_v)_{v\in\Omega_k}$. Such $S$ exists by a limit argument. We claim that $A_{n,S}\subseteq\zcyclef{}{\vA}(X)^B$. 

Take an adelic 0-cycle $(x_v)_{v\in\Omega_k}\in A_{n,S}$, then $x_v$ is rationally equivalent to some 0-cycle with from $n.x_v'$ on $X_v$ for each $v\in S$ by definition. Hence, we have \begin{equation*}
\BMcouple{\alpha,(x_v)_{v\in\Omega_k}}=\sum\limits_{v\in S}\inv_v(\alpha,x_v)=\sum\limits_{v\in S}\inv_v(n.\alpha,x_v')
\end{equation*}for each $\alpha\in B$. Since $n.\alpha\in\br_0(X)$, this sum is 0. Therefore, $A_{n,S}$ is orthogonal to $B$.

To show the first assersion that $\zcyclef{}{\vA}(X)^B$ is always closed, we notice that \begin{equation*}
\zcyclef{}{\vA}(X)^B=\bigcap\limits_{\alpha\in B}\zcyclef{}{\vA}(X)^{\{\alpha\}}.
\end{equation*}Since $\alpha$ is vanished by some positive integer and an open subgroup is always closed in the topology of $\zcyclef{}{\vA}(X)$, the first assersion follows from the second assersion. 
\end{proof}

To reveal the $f$-descent decomposition  for 0-cycles, we care the extension of the ground field $k$.

Let $K$ be a finite extension of $k$. Let $X_K(\vA)$ and $\zcyclef{}{\vA}(X_K)$ denote the associated sets of adelic points and adelic 0-cycles over $K$-variety $X_K$. Given a place $w\in\Omega_K$ lying above $v\in\Omega_k$, we have the pushforward map $\pi^{w|v}_{K/k}:\zcycle{}(X_w)\rightarrow\zcycle{}(X_v)$ for 0-cycles associated to projection $X_w\rightarrow X_v$. Here $X_w$ means $X_{K_w}$. Let $\pi_{K/k}:\zcyclef{}{\vA}(X_K)\rightarrow\zcyclef{}{\vA}(X)$ be the pushforward homomorphism of adelic 0-cycles determined by all these \{$\pi_{w|v}$\}. Precisely, that means $\pi_{K/k}((x_w)_{w\in\Omega_K})=\Big(\sum\limits_{w|v}\pi_{K/k}^{w|v}(x_w)\Big)_{v\in\Omega_k}$ for any adelic 0-cycle $(x_w)_{w\in\Omega_K}\in\zcyclef{}{\vA}(X_K)$.

\subsection{Decomposition of the abelianized descent obstruction for 0-cycles}\,

We recall the definition of $\zcyclef{}{\vA}(X)^f_\BB$.

\begin{definition}\label{def-desc-1}
Let $X$ be a smooth and geometrically integral variety over a number field $k$, $G$ a linear $k$-group and $f:Y\rightarrow X$ a torsor under $G$. Let $\delta$ be an integer.

(1) Say that an adelic 0-cycle $(x_v)_{v\in\Omega_k}\in\zcyclef{\delta}{\vA}(X)$ is \textbf{$f$-descent by an adelic 0-cycle} if there is a finite extension $K$ of $k$ and an element $\sigma\in\cohomorf{1}{K,G}$, such that $(x_v)_{v\in\Omega_k}=\pi_{K/k}(f^\sigma_{K,*}((y_w)_{w\in\Omega_K}))$ for an adelic 0-cycle $(y_w)_{w\in\Omega_K}$ on $Y_K^\sigma$. Alternatively speaking, it lies in the image of the following composition of homomorphisms\begin{equation*}
\xymatrix{\zcyclef{}{\vA}(Y_K^\sigma)\ar[r]^{f^\sigma_{K,*}}&\zcyclef{}{\vA}(X_K)\ar[r]^{\pi_{K/k}}&\zcyclef{}{\vA}(X)}.
\end{equation*}


(2) The $f$-descent subgroup $\zcyclef{}{\vA}(X)^f_{\BB}$ defined in \cite{BB24} is no other than the subgroup of $\zcyclef{}{\vA}(X)$, generated by all adelic 0-cycles on $X$, $f$-descent by an adelic 0-cycle.


\end{definition}

	The following proposition shows $\zcyclef{}{\vA}(X)^f_\BB$ is stricter than $\zcyclef{}{\vA}(X)^f_\ab$.

	\begin{proposition}\label{prop1} Let $X$ be a smooth, proper and geometrically integral variety over a number field $k$, $G$ a connected linear $k$-group and $f:Y\rightarrow X$ a torsor under $G$, then we obtain that \begin{equation*}
		\zcyclef{}\vA(X)_{\mathrm{BB}}^f\subseteq\zcyclef{}\vA(X)_\ab^f.
		\end{equation*}
	\end{proposition}
	
	\begin{proof}We only need to claim that any adelic 0-cycle $(x_v)_{v\in\Omega_k}$, $f$-descent by an adelic 0-cycle, lies in $\zcyclef{}\vA(X)_\ab^f$. Assume that $(x_v)_{v\in\Omega_k}$ has form $\pi_{K/k}(f_{K,*}^\sigma((y_w)_{w\in\Omega_K}))$ for a finite extension $K/k$ and an element $\sigma\in\cohomorf{1}{K,G}$. Take $\alpha\in\br_f'\subseteq\br(X)$, then one has \begin{equation*}
	\BMcouple{\alpha,(x_v)_{v\in\Omega_k}}{_{,k}}=\BMcouple{\beta,(y_w)_{w\in\Omega_K}}{_{,K}}
	\end{equation*}where $\beta$ is the image of $\alpha$ along $\br(X)\rightarrow\br(X_K)\xrightarrow{f_K^{\sigma,\ast}}\br(Y_K^\sigma)$ by functoriality of \BM\ pairing. Since the restriction $\alpha_K\in\br(X_K)$ lies in $\br_{f_K}'$, the image $\beta$ lies in $\br_0(Y_K^\sigma)$ by Lemma \ref{brf-dec}, hence, $\BMcouple{\beta,(y_w)_{w\in\Omega_K}}{_{,K}}=0$. Now, $(x_v)_{v\in\Omega_k}$ is orthogonal to $\br_f'$ and lies in $\zcyclef{}{\vA}(X)^f_\ab$ by Theorem \ref{main2}.
	\end{proof}

	Before we give the converse of Proposition \ref{prop1}, we give the following lemma at first.
	
	\begin{lemma}\label{lem_mod}
	Let $X$ be a smooth, proper and geometrically integral variety over a number field $k$. Assume that following assumptions are satisfied:
	
	(a) $\pic{X_{\overline k}}$ is of finite type;
	
	(b) for each positive integer $d$, there exists a finite extension $k'$ of $k$ such that for any finite extension $K$ of $k$ of degree $d$ that is linearly disjoint from $k'$ over $k$, the canonical morphism $\br(X)/\br_0(X)\twoheadrightarrow\br(X_K)/\br_0(X_K)$ is surjective.
	
	Then for any connected linear $k$-group $G$ and a torsor $f:Y\rightarrow X$ under $G$ and any positive integer $d$, there exists a finite extension $k'$ (dependent on $f$) of $k$ such that for any finite extension $K$ of $k$ of degree $d$ that is linearly disjoint from $k'$ over $k$, the canonical morphism \begin{equation*}
	(\br_f'+\br_0(X))/\br_0(X)\twoheadrightarrow(\br_{f_K}'+\br_0(X_K))/\br_0(X_K)
	\end{equation*}is surjective.
	\end{lemma}
	
	\begin{proof}
	Since the Brauer group $\br(X)$ is invariant up to a birational equivalence, hence, by Chow's lemma, we may assume that $X$ is projective. By Hironaka's resolution of singularity, we may find a projective morphism $f':Z\rightarrow X$ such that $Y$ is an open dense subset of $Z$ and $f$ is the restriction of $f'$ on $Y$. Hence, $Z$ is projective and geometrically integral over $k$. 
	
	We first claim that $\pic{Z_{\overline k}}$ is of finite type. By Sansuc's exact sequence (see \cite[Proposition 6.10]{Sansuc81})
	\begin{equation*}
	\pic{X_{\overline k}}\rightarrow\pic{Y_{\overline k}}\rightarrow\pic{G_{\overline k}},
	\end{equation*}we know that $\pic{G_{\overline k}}$ is of finite type, hence, $\pic{Y_{\overline k}}$ is of finite type. Therefore, $\pic{Z_{\overline k}}$ is of finite type by the exact sequence
	\begin{equation*}
	\bigoplus\limits_{y\in(Z_{\overline k}\backslash Y_{\overline{k}})^{(1)}}\mathbb{Z}\rightarrow\pic{Z_{\overline k}}\rightarrow\pic{Y_{\overline k}}\rightarrow 0.
	\end{equation*}
	
	Now, let $d$ be an arbitrary positive integer and $k'$ be a finite extension of $k$ satisfying the assumption (b).
	
	Let $k''$ be a finite extension of $k'$ large enough such that $\pic{Z_{k''}}=\pic{Z_{\overline k}}$. By construction, $\Gamma_{k''}$ acts trivially on $\pic{Z_{\overline k}}$, hence, one obtains that $\cohomorf{1}{k'',\pic{Z_{\overline k}}}=0$ and $\cohomorf{1}{\Gamma_{k''/k},\pic{Z_{\overline k}}}=\cohomorf{1}{k,\pic{Z_{\overline k}}}$ by the inflation-restriction exact sequence (see for example, \cite[Proposition 4, Section 6, Chapter VII]{Ser62}). We claim that the field $k''$ is the required field.
	 
	By spectral sequence \begin{equation*}
	\cohomorf{p}{k,\cohomorf{q}{Z_{\overline k},\bG_m}}\Rightarrow\cohomorf{p+q}{Z,\bG_m},
	\end{equation*}we have an exact sequence \begin{equation*}
	0\rightarrow\cohomorf{1}{k,\pic{Z_{\overline k}}}\rightarrow\br(Z)/\br_0(Z)\rightarrow\br(Z_{\overline k}).
	\end{equation*}This sequence is functorial at $k$. 
	
	Let $K$ be a finite extension of $k$ of degree $d$ that is linearly disjoint from $k''$ over $k$ and $K''$ be the composite of $K$ and $k''$, then $\Gamma_{K''/K}\xrightarrow\sim\Gamma_{k''/k}$ and $\pic{Z_{k''}}=\pic{Z_{K''}}=\pic{Z_{\overline k}}$, we also see that $\cohomorf{1}{\Gamma_{K''/K},\pic{Z_{\overline K}}}=\cohomorf{1}{K,\pic{Z_{\overline K}}}$. Therefore, we obtain a commutative diagram by functoriallity with exact rows\begin{equation*}
	\xymatrix{0\ar[r]&\cohomorf{1}{\Gamma_{k''/k},\pic{Z_{k''}}}\ar@{=}[d]\ar[r]&\br(Z)/\br_0(Z)\ar[r]\ar[d]&\br(Z_{\overline k})\ar@{=}[d]\\
	0\ar[r]&\cohomorf{1}{\Gamma_{K''/K},\pic{Z_{K''}}}\ar[r]&\br(Z_K)/\br_0(Z_K)\ar[r]&\br(Z_{\overline K})}.
	\end{equation*}This shows that $\br(Z)/\br_0(Z)\rightarrow\br(Z_K)/\br_0(Z_K)$ is injective.
	
	Since $\br(Z)\rightarrow \br(Y)$ is injective and $K$ is also linearly disjoint from $k'$ over $k$, we have a commutative diagram\begin{equation*}
	\xymatrix{0\ar[r]&(\br_f'+\br_0(X))/\br_0(X)\ar[r]\ar[d]&\br(X)/\br_0(X)\ar[r]\ar@{->>}[d]&\br(Z)/\br_0(Z)\ar@{>->}[d]\\
	0\ar[r]&(\br_{f_K}'+\br_0(X_K))/\br_0(X_K)\ar[r]&\br(X_K)/\br_0(X_K)\ar[r]&\br(Z_K)/\br_0(Z_K)},
	\end{equation*}the lemma is claimed by a diagram chasing.
	\end{proof}

	\begin{theorem}\label{main-geo-general} Let $X$ be a smooth, projective and geometrically integral variety over a number field $k$, $G$ a connected linear $k$-group and $f:Y\rightarrow X$ a torsor under $G$. Assume that following assumptions are satisfied:
			
		(a) the abelianized $f$-descent subgroup $\zcyclef{}{\vA}(X)^f_\ab$ is open in $\zcyclef{}{\vA}(X)$;
		
		(b) $\pic{X_{\overline k}}$ is of finite type;
		
		(c) for each positive integer $d$, there exists a finite extension $k'$ of $k$ such that for any finite extension $K$ of $k$ of degree $d$ that is linearly disjoint from $k'$ over $k$, the canonical morphism $\br(X)/\br_0(X)\twoheadrightarrow\br(X_K)/\br_0(X_K)$ is surjective.
		
		Then we obtain that 	
					\begin{equation*}
					\zcyclef{}\vA(X)_\ab^f=\overline{\zcyclef{}\vA(X)_{\mathrm{BB}}^f}.
					\end{equation*}
		\end{theorem}
		
	\begin{proof}Since $\zcyclef{}{\vA}(X)_\ab^f$ is always closed in $\zcyclef{}{\vA}(X)$ by Proposition \ref{closed-ness}, we obtain that $\overline{\zcyclef{}\vA(X)_{\mathrm{BB}}^f}\subseteq\zcyclef{}\vA(X)_\ab^f$ by Proposition \ref{prop1}. Hence, we only need to claim that \begin{equation*}
	\zcyclef{}\vA(X)_\ab^f\subseteq\bigcap\limits_{n,S}(\zcyclef{}\vA(X)_{\mathrm{BB}}^f+A_{n,S}).
	\end{equation*} 
	
	Let $(x_v)_{v\in\Omega_k}$ be an adelic 0-cycle in $\zcyclef{}{\vA}(X)_\ab^f$. By assumption (a), there is a positive integer $n_0$ and a finite subset $S_0$ of $\Omega_k$ such that $A_{n_0,S_0}\subseteq\zcyclef{}{\vA}(X)^f_\ab$, therefore, we only need to claim that for any positive multiple $n$ of $n_0$ and any finite subset $S$ of $\Omega_k$ containing $S_0$, $(x_v)_{v\in\Omega_k}$ has the following decomposition
	\begin{equation*}
	(x_v)_{v\in\Omega_k}=x+\pi_{K/k}((Q_w)_{w\in\Omega_K})+(z_v)_{v\in\Omega_k}
	\end{equation*}where $x$ is a global 0-cycle on $X$, $(z_v)_{v\in\Omega_k}\in A_{n,S}$, $(Q_w)_{w\in\Omega_K}$ is an adelic point on $X_K$ and $K$ is a finite extension of $k$ such that the canonical morphism \begin{equation*}
	(\br_f'+\br_0(X))/\br_0(X)\twoheadrightarrow(\br_{f_K}'+\br_0(X_K))/\br_0(X_K)
	\end{equation*}is surjective. 
	
	We first mention that such an adelic point $(Q_w)_{w\in\Omega_K}$ is orthogonal to $\br'_{f_K}$, hence, $(Q_w)_{w\in\Omega_K}$ lies in the image of $f_K^\sigma:Y_K^\sigma(\vA)\rightarrow X_K(\vA)$ for some $\sigma\in\cohomorf{1}{K,G}$ and $\pi_{K/k}((Q_w)_{w\in\Omega_K})$ is $f$-descent by an adelic point on $Y_{K_i}^\sigma$. In fact, given arbitrary $\alpha_K\in\br_{f_K}'$, one obtains that $\alpha_K=\alpha|_K+\beta$ where $\alpha\in\br_f'$, $\alpha|_K$ is the base-change of $\alpha$ in $\br(X_K)$ and $\beta\in\br_0(X_K)$. We see that \begin{equation*}
	\BMcouple{\alpha_K,(Q_w)_{w\in\Omega_K}}{_{,K}}=\BMcouple{\alpha,\pi_{K/k}((Q_w)_{w\in\Omega_K})}{_{,k}}.
	\end{equation*}Since $\pi_{K/k}((Q_w)_{w\in\Omega_K})=(x_v)_{v\in\Omega_k}-x-(z_v)_{v\in\Omega_k}$ is orthogonal to $\br_f'$, one obtains that $(Q_w)_{w\in\Omega_K}$ is orthogonal to $\br'_{f_K}$.
	
Now, fix a closed point $P$ on $X$. We may assume that $S$ contains all archimedean places and that $X(k_v)\not=\emptyset$ for each $v\not\in S$ by Lang's estimate. Consider the trivial fibration $h:Y=X\times_k\bP^1_k\rightarrow\bP_k^1$. Fix a rational point $0\in\bP^1_k(k)$ and an isomorphism $\bA^1_k=\spec{k[t]}\simeq\bP^1_k\backslash$\{$0$\}. For each $v\in\Omega_k$, we have the 0-cycle $y_v=x_v\times 0$ and the global 0-cycle $P_0=P\times 0$ on $Y_v$. 


By \cite[Lemma 3.2]{CT00}, there is a positive integer $r$ such that for each $v\in S$, $y_v+r.P_0$ is rationally equivalent to an effective 0-cycle $y_v^1$ and its pushforward $f_*(y_v^1)$ is ``separable'' and supported in $\bA^1_v\subsetneq\bP^1_v$. A ``separable'' 0-cycle means a sum of distinct closed points (it is effective). Each $f_*(y_v^1)$ defines a separable monic polynomial $R_v(t)\in k_v[t]$ of degree $\delta_r=\delta+r.\deg(P)>0$. 

Let $k'$ be a finite extension of $k$ such that for each finite extension $K$ of $k$ of degree $\delta_r$ that is linearly disjoint from $k'$ over $k$, the restriction homomorphism \begin{equation*}
	(\br_f'+\br_0(X))/\br_0(X)\twoheadrightarrow(\br_{f_K}'+\br_0(X_K))/\br_0(X_K)
	\end{equation*}is surjective. The projection map $\bP^1_{k'}\rightarrow\bP^1$ defines a Hilbert subset $\textbf{\textit{Hil}}$ on $\bP^1$. By \cite[Proposition 4.7]{Liang23} (take $y_\infty$ to be $\delta_r.\infty$ and $z_v$ to be $f_*(y_v^1)$ there), there is a closed point $\theta\in\textbf{\textit{Hil}}$ such that for each $v\in S$, $\theta$ is rationally equivalent to $f_*(y_v^1)$ and $\theta$ can be chosen arbitrarily close to $f_*(y_v^1)$ if we identify effective 0-cycles of degree $\delta_r$ on $\bP^1$ as rational points on the symmetric product $\text{Sym}^{\delta_r}_{\bP^1/k}(k_v)$.


For each $v\in S$, since the induced morphism $h^{\delta_r}:\text{Sym}_{Y/k}^{\delta_r}(k_v)\rightarrow\text{Sym}_{\bP^1/k}^{\delta_r}(k_v)$ is smooth at $y_v^1$ (treated as a rational point on $\text{Sym}_{Y_v/k_v}^{\delta_r}$), by implicit function theorem, there is an effective 0-cycle $y_v^2$ on $Y_v$ such that $y_v^2$ is arbitrarily closed to $y_v^1$ as rational point of $\text{Sym}_{Y/k}^{\delta_r}(k_v)$ and $f_*(y_v^2)=\theta\times k_v$ on $\bP^1_v$. By \cite[Lemma 1.8]{Wit12}, adelic 0-cycles $y_v^2$ and $y_v^1$ has the same image in $\chow_0(X_v)/n$.


Now, let $K=k(\theta)$, then for each $v\in S$, the equation $f_*(y_v^2)=\theta\times k_v=\coprod\limits_{w|v}\spec{K_w}$ shows that  $y_v^2=\sum\limits_{w|v}Q_w$ where $Q_w$ is a rational point on $X_w$. For each $w\not\in S_K$, there is a rational point  $Q_w\in X_w(K_w)$ by the choice of $S$. We are going to claim that $(Q_w)_{w\in\Omega_K}$ is the required adelic point.

Remember that $y_v+r.P_0$ and $y_v^1$ are rationally equivalent on $Y_v$, then $x_v+r.P$ and $\sum\limits_{w|v}\pi_{w|v}(Q_w)$ has the same image in $\chow_0(X_v)/n$. Therefore, the adelic 0-cycle $(x_v)_{v\in\Omega_k}+r.P-\pi_{K/k}((Q_w)_{w\in\Omega_k})$ lies in $A_{n,S}$. Then the lemma follows.
	\end{proof}

\begin{corollary}\label{main-geo} Let $X$ be a smooth, projective and geometrically integral variety over a number field $k$, $G$ a connected linear $k$-group and $f:Y\rightarrow X$ a torsor under $G$. In each of the following cases:

(i) $X$ is rationally connected,

(ii) $X$ is a K3 surface,

we obtain that 	
					\begin{equation*}
					\zcyclef{}\vA(X)_\ab^f=\overline{\zcyclef{}\vA(X)_{\mathrm{BB}}^f}.
					\end{equation*}
		\end{corollary}
		
\begin{proof}
We only need to check three assumptions in Proposition \ref{main-geo-general} for each case.

In both cases, $\br(X)/\br_0(X)$ is finite, hence, assumption (a) is guaranteed by Proposition \ref{closed-ness}. For finiteness of $\br(X)/\br_0(X)$, see \cite[Proposition 3.1.1]{Liang13} or \cite[Theorem 1.2]{SZ08} respectively. 

In both cases, $\pic{X_{\overline k}}$ is torsion-free of finite type. See \cite[Proposition 1.4.2]{Liang13} or \cite[Proposition 2.4, Chapter 1]{Huy16} respectively. 

The assumption (c) is also guaranteed for both cases, see \cite[Proposition 3.1.1]{Liang13} or \cite[Theorem 1.2]{Ier21} respectively.
\end{proof}

\begin{remark}
If we denote $\zcyclef{}{\vA}(X)^{\mathrm{conn}}_\BB=\bigcap\limits_{f:Y\xrightarrow GX}\zcyclef{}{\vA}(X)^f_\BB$, where the intersection is taken over all $X$-torsors $f:Y\xrightarrow{G}X$ under all connected linear groups $G$ over $k$. One can immediately show that $\zcyclef{}{\vA}(X)^{\mathrm{conn}}_\BB\subseteq\zcyclef{}{\vA}(X)^\br$ with the help of Proposition \ref{prop1}. But the author does not know whether $\overline{\zcyclef{}{\vA}(X)^{\mathrm{conn}}_\BB}=\zcyclef{}{\vA}(X)^\br$ holds under suitable assumptions.
\end{remark}

\subsection{Universal torsors under tori.}\,


Kollar in \cite[Theorem 5]{Kol03} proved that for any smooth, projective and rationally connected variety $X$ over a number field $k$, there is a finite subset $S\subset\Omega_k$ such that $\chow^0_0(X_v)=0$ for any $v\not\in S$. In \cite[Proposition 11]{CT05}, \CT\ pointed out that $\chow_0^0(X_v)$ is annihilated by a uniform positive integer $N$ (independent on the choice of $v\in\Omega_k$). On the other hand, if $v$ is a real place, then $\chow_0^0(X_v)=(\mathbb{Z}/2)^{\text{max}(s-1,0)}$ where $s\geq0$ is the number of connected components of the toplogical space $X(k_v)$. With these evidences, \CT\ gave the following conjecture in \cite[Section 5]{CT05}.

\begin{conjecture}\label{conj}Assume that $X$ is a smooth, projective and rationally connected variety over a $p$-adic field $k$, then $\chow_0^0(X)$ is a finite group.
\end{conjecture}

If Conjecture \ref{conj} holds, then for any smooth projective and rationally connected variety $X$ over a number field $k$, group $\chow_{0,\vA}(X)$ is an abelian group of finite type due to the following exact sequence\begin{equation*}
0\rightarrow\chow_{0,\vA}^0(X)\rightarrow\chow_{0,\vA}(X)\xrightarrow{\text{deg}}\bZ.
\end{equation*}Hence, it is natural to study the behaviour of 0-cycles on a variety $X$ with a finite abelian group $\chow^0_{0,\vA}(X)$.

\begin{lemma}\label{chow00}
Let $X$ be a smooth, projective and geometrically integral variety over a number field $k$, assume that $\chow^0_{0,\vA}(X)$ is a finite abelian group, then 

(1) there exists a positive integer $N$ and a finite subset $S$ of $\Omega_k$ such that $A_{N,S}$ is a minimal open subgroup. This means that for any positive multiple $N'$ of $N$ and a finite subset $S'\supseteq S$, one obtains  that $A_{N',S'}=A_{N,S}$;

(2) subgroup $\zcyclef{}{\vA}(X)^\br$ is an open subgroup of $\zcyclef{}{\vA}(X)$.
\end{lemma}

\begin{remark}
Statement (1) of Lemma \ref{chow00} tells us that, to verify whether $X$ satisfies weak approximation with the \BM\ obstruction for 0-cycles of degree $\delta$, one only need to check that $\zcyclef{\delta}{\vA}(X)^\br=\zcycle{\delta}(X)+A_{N,S}$. Here $N,S$ are chosen as in the statement (1). 
\end{remark}

\begin{proof}
Claim (1). Since $\chow^0_{0,\vA}(X)=\prod\limits_{v\in\Omega_k}\chow_{0}^0(X_v)$ is finite, there is a finite subset $S$ of $\Omega_k$ such that for each $v\not\in S$, $\chow_0^0(X_v)=0$ and a positive integer $N$ such that $N.\chow^0_{0,\vA}(X)=0$. Hence, an adelic 0-cycle $(x_v)_{v\in\Omega_k}$ lies in $A_{N,S}$ if and only if $x_v$ vanishes in $\chow_0^0(X_v)$ for each $v\in\Omega_k$. This means that we have an exact sequence\begin{equation*}
0\rightarrow A_{N,S}\rightarrow \zcyclef{}{\vA}(X)\rightarrow\chow_{0,\vA}(X)\rightarrow0.
\end{equation*}

Now, for any  positive multiple $N'$ of $N$ and any finite subset $S'\supseteq S$ of $\Omega_k$, it is clear that $A_{N,S}\subseteq A_{N',S'}$ and immediately, they are equal. 


Claim (2). Choose $S,N$ as above. Since an adelic point $(x_v)_{v\in\Omega_k}\in A_{N,S}$ is rationally equivalent to $0$, it is orthogonal to $\brgp{X}$. Hence, $A_{N,S}\subseteq\zcyclef{}{\vA}(X)^\br$. A group containing an open subgroup is also open, then $\zcyclef{}{\vA}(X)^\br$ is open.
\end{proof}

\begin{theorem}
Let $X$ be a smooth, projective and geometrically integral variety over a number field $k$ such that $\pic{X_{\overline k}}$ is torsion free and $\zcyclef{1}{\vA}(X)^{\br_1}\neq\emptyset$. Assume that $\zcyclef{}{\vA}(X)^\br$ is open subgroup of $\zcyclef{}{\vA}(X)$ and $\chow_{0,\vA}(X)$ is abelian group of finite type, then there are finitely many field extensions \{$K_1,...,K_n$\} of $k$ and for each $i$ a universal torsor $f_i:Y_i\rightarrow X_{K_i}$ over $X_{K_i}$ such that, if  $Y_i$ satisfies the weak approximation for rational points for each $i$, then $X$ satisfies the weak approximation with the \BM\ obstruction for 0-cycles. 
\end{theorem}

\begin{proof}By Theorem \ref{cor1}, a universal torsor $f:Y\rightarrow X$ exists.
Since $\chow^0_{0,\vA}(X)$ is finite, $\chow_{0,\vA}(X)$ is a abelian group of finite type, so does the image $A$ of $\zcyclef{}{\vA}(X)^{\br}$ in $\chow_{0,\vA}(X)$. Let \{$(x_{1,v})_{v\in\Omega_k},...,(x_{n,v})_{v\in\Omega_k}$\} be a finite subset of $\zcyclef{}{\vA}(X)^\br$ such that their images in $\chow_{0,\vA}(X)$ generate $A$. We choose a finite subset $S$ and a positive integer $N$ as in Lemma \ref{chow00}.


Since $\zcyclef{}{\vA}(X)^f_\BB$ is dense in $\zcyclef{}{\vA}(X)^f_\ab=\zcyclef{}{\vA}(X)^{\br_1}$ by Theorem \ref{main-geo}. For each $i$, $(x_{i,v})_{v\in\Omega_k}$ can be written as the form:
\begin{equation*}
(x_{i,v})_{v\in\Omega_k}=x_i+\pi_{K_i/k}(f_{K_i,*}^{\sigma_i}((Q_{i,w_i})_{w_i\in\Omega_{K_i}}))+(z_{i,v})_{v\in\Omega_k}
\end{equation*}where $x_i$ is a global 0-cycle on $X$, $K_i$ is a finite extension of $k$, $f_{K_i}^{\sigma_i}$ is the twist of $f_{K_i}:Y_{K_i}\rightarrow X_{K_i}$ by some $\sigma_i\in\cohomorf{1}{K_i,G}$, $(Q_{i,w_i})_{w_i\in\Omega_{K_i}}$ is an adelic point on $Y^{\sigma_i}_{K_i}$ and $(z_{i,v})_{v\in\Omega_k}$ lies in $A_{N,S}$. 


Now we claim that if for each $i$, universal torsors $Y^{\sigma_i}_{K_i}$ satisfy the weak approximation for rational points, then $X$ satisfies the weak approximation for 0-cycles. In fact, let $(x_v)_{v\in\Omega_k}\in\zcyclef{}{\vA}(X)^\br\subseteq\zcyclef{}{\vA}(X)^{\br_1}$ be an arbitrary adelic 0-cycle. By the choice of \{$(x_{i,v})_{v\in\Omega_k}$\}, the adelic 0-cycle $(x_v)_{v\in\Omega_k}$ can be written as the form: \begin{equation*}
(x_v)_{v\in\Omega_k}=\sum\limits_{1\leq i\leq n}r_i.(x_{i,v})_{v\in\Omega_k}+(z_v)_{v\in\Omega_k}
\end{equation*}where $\{r_i\}$ are integers,  and $(z_v)_{v\in\Omega_k}$ lies in $\kernel{\zcyclef{0}{\vA}(X)\rightarrow\chow_{0,\vA}^0(X)}= A_{N,S}$. 

For each $i$, according to Hironaka's resolution of singularity, we take a projective morphism $g_i:Z_i\rightarrow X_{K_i}$ such that $Y_{K_i}^{\sigma_i}$ is a dense open subset of $Z_i$ and $f_{K_i}^{\sigma_i}$ is the restriction of $g_i$ on $Y_{K_i}^{\sigma_i}$.
Since $Y_{K_i}^{\sigma_i}$ satisfies the weak approximation for rational points for each $i$, by Lemma \cite[Lemma 1.8]{Wit12}, there is a ratioanl point $Q_i$ of $Z_i$ such that $Q_i$ and $(Q_{i,w_i})_{w_i\in\Omega_{K_i}}$ have the same image in $\chow_0(Z_{i,w})/N$ for each $w$ lies above some $v\in S$. Hence, $\pi_{K_i/k}({f_{K_i,*}^{\sigma_i}}(Q_i))-\pi_{K_i/k}({f_{K_i,*}^{\sigma_i}}((Q_{i,w_i})_{w_i\in\Omega_{K_i}}))=\pi_{K_i/k}({g_{i,*}(Q_i-(Q_{i,w_i})_{w_i\in\Omega_{K_i}}}))$ lies in $A_{N,S}$.

Put them together, we obtain that \begin{equation*}
(x_v)_{v\in\Omega_k}=\underbrace{\Big(\sum\limits_{1\leq i\leq n}r_i.x_i\Big)+\Big(\sum\limits_{1\leq i\leq n}r_i.\pi_{K_i/k}(f_{K_i,*}^{\sigma_i}(Q_i))\Big)}_{\text{a global 0-cycle}}+(z_v')_{v\in\Omega_k}.
\end{equation*}Here $(z_v')_{v\in\Omega_k}$ lies in $A_{N,S}$ (it is a finite sum of adelic 0-cycles in $A_{N,S}$). This completes the proof.
\end{proof}

	\vspace{3ex}
		\textbf{Funding.} This work was supported by Quantum Science and Technology-National Science and Technology Major Project [2021ZD0302902].
		
	\vspace{3ex}
			\textbf{Acknowledgements.} The author would like to thank my PhD advisor Yongqi Liang, Prof. Yang Cao and Prof. Yisheng Tian for their kind suggestions.	


	\newpage
	\printbibliography

\end{document}